\documentclass[]{article}

\usepackage{amsmath,amssymb,amsthm}
\usepackage{hyperref}
\usepackage{tikz}
\usepackage{algorithm}
\usepackage{algpseudocode}

\usetikzlibrary{intersections,patterns,decorations.pathreplacing}

\newcommand{\Z}{\mathbb{Z}}
\newcommand{\N}{\mathbb{N}}
\newcommand{\lang}{\mathcal{L}}
\newcommand{\set}[2]{\left\{ #1 \;\middle|\; #2 \right\}}

\newcommand{\agrid}{A_{\mathrm{grid}}}
\newcommand{\grid}{X_{\mathrm{grid}}}
\newcommand{\xgrid}{x_{\mathrm{grid}}}
\newcommand{\acnt}{A_{\mathrm{count}}}
\newcommand{\cnt}{X_{\mathrm{count}}}
\newcommand{\aprb}{A_{\mathrm{probe}}}

\theoremstyle{plain}
\newtheorem{lemma}{Lemma}
\newtheorem{corollary}[lemma]{Corollary}
\newtheorem{proposition}[lemma]{Proposition}
\newtheorem{theorem}[lemma]{Theorem}
\newtheorem{conjecture}[lemma]{Conjecture}
\newtheorem{question}[lemma]{Question}

\theoremstyle{example}

\theoremstyle{definition}
\newtheorem{definition}[lemma]{Definition}
\newtheorem{example}[lemma]{Example}

\title{On Countable SFT Covers of Sparse Multidimensional Shift Spaces}
\author{Ilkka Törmä\footnote{Author was supported by the Academy of Finland under grant 346566.} \\ Department of Mathematics and Statistics \\ University of Turku, Turku, Finland \\ \texttt{iatorm@utu.fi}}

\begin{document}

\maketitle

\begin{abstract}
    A multidimensional sofic shift is called countably covered if it has an SFT cover containing only countably many configurations.
    In contrast to the one-dimensional setting, not all countable sofic shifts are countably covered.
    We investigate the existence of countable covers for gap width shifts, where the number of nonzero symbols in a configuration is bounded by a function of the minimum distance between two such symbols.
    As our main results, we characterize those one-dimensional gap width shifts whose two-dimensional lift is a countably covered sofic shift, and show that a large class of two-dimensional gap width shifts are countably covered.
\end{abstract}

\section{Introduction}

A multidimensional shift of finite type (SFT) is a set of colorings of $\Z^d$ using finitely many colors that is defined by a finite number of local rules.
A sofic shift is the image of an SFT, called its cover, under a map that identifies some of the colors with each other.
We study the following problem: given a two-dimensional sofic shift whose cardinality is countable, when does it admit an SFT cover that is also countable?
More generally, we may consider the problem of finding countable-to-one covers for arbitrary sofic shifts.

We concentrate on the case $d = 2$ and the case where the sofic shift is in some sense ``sparse'', that is, the alphabet contains $0$ and each configuration contains few nonzero symbols: either their number is finite, or they lie on finitely many columns.
In the latter case we also constrain each column to be uniform, so that each configuration is $(0,1)$-periodic and the number of configurations is necessarily countable.

The existence of a global periodic direction is the setting in which the author previously proved in \cite{To20} that not all countable sofic shifts have countable SFT covers.
The author proved a necessary condition and showed it to be sufficient it in a restricted subclass of sofic shifts.
In this article we prove the same statement in another subclass: those in which all but finitely many columns have the same ``zero'' color, and the number of nonzero columns is restricted by a computable function of the size of the smallest gap between two nonzero columns.
In particular, we can prove that not all sofic shifts of this form admit countable SFT covers.

A two-dimensional shift space with $(0,1)$ as a global period is the \emph{lift} of a one-dimensional shift space.
It is a well known result of Aubrun-Sablik \cite{AuSa13} and Durand-Romashchenko-Shen \cite{DuRoSh12} that the $\Z^2$-lift of a $\Z$-shift space is sofic, i.e.\ has an SFT cover, if and only if it is effectively closed.
In \cite{DuRo21}, Durand and Romashchenko further prove that if the $\Z$-shift space is quasiperiodic, then the SFT cover can be taken quasiperiodic as well, and the same holds for minimality.
One can say that quasiperiodicity and minimality are properties that can always be lifted to the SFT cover of a vertically constant sofic shift.
Countability, as it turns out, cannot always be lifted in this way.

On the other hand, this study is also related to the \emph{equal entropy problem} of whether every $\Z^2$-sofic shift has an SFT cover of equal topological entropy.
Desai showed in \cite{De06} that there always exist covers whose entropy is arbitrarily close to that of the sofic shift, but reaching the exact value remains an open question.
In particular, it is not known whether every zero entropy sofic shift admits a zero entropy SFT cover, that is, whether having zero entropy is a liftable property.
Countability, like having zero entropy, is a ``smallness property'' of shift spaces, whose liftability we investigate.

Our second main result concerns truly sparse two-dimensional shift spaces, those for which the total number of nonzero symbols is bounded by a function of the minimum distance between two such symbols.
We show that if there are no other constraints, and the function is upper semicomputable and bounded from above by $2^{\sqrt{n}}$, then the resulting shift space is a countably covered sofic shift.
We do not provide examples of sparse two-dimensional sofic shifts that are not countably covered, as there are no known techniques for proving this property in the absence of a periodic direction.

These results are proved using a general toolbox of constructions for implementing computations and geometric or combinatorial constraints in countable shifts of finite type.
We have presented them in a modular form for ease of reuse in other constructions of countable SFTs.
Some of them may also be useful in constructing countable-to-one covers for general sofic shifts, although this is beyond the scope of this article.

The two main results, despite their seeming similarity, require subtly different approaches.
The existence of a periodic direction (or rather, only having one non-periodic dimension) is both a blessing and a curse.
On one hand, every finite pattern occurs in infinitely many positions, which allows us to ignore any given finite region of the configuration and check its contents somewhere else.
On the other hand, every periodic configuration must have a periodic preimage in the SFT cover \cite[Proposition~1]{To20}, so we cannot enforce infinite computations in the cover.
In both cases, the requirement of countability means that we cannot employ any of the known constructions of aperiodic SFTs, such as Robinson tiles~\cite{Ro71} or fixed-point tiles~\cite{DuRoSh12}, except in carefully controlled ways like restricting them to a finite region.

This article is dedicated to prof.\ Jarkko Kari, my former PhD advisor, on his 60th birthday.
I am thankful to him for many years of collaboration and friendship.
I am also thankful to the anonymous referee for their comments that helped significantly improve this article.

\section{Preliminaries}

\subsection{Definitions}

We recall some basic notions of multidimensional symbolic dynamics.
Let $A$ be a finite alphabet containing $0$ and $d \geq 1$ a dimension.
Denote $A_0 = A \setminus \{0\}$.
The $d$-dimensional \emph{full shift} over $A$ is the set $A^{\Z^d}$, whose element are called \emph{configurations}.
It is equipped with the \emph{shift action} of $\Z^d$, defined by $\sigma^{\vec v}(x)_{\vec n} = x_{\vec n + \vec v}$ for $x \in A^{\Z^d}$ and $\vec v, \vec n \in \Z^d$.
A $d$-dimensional \emph{pattern} with finite \emph{domain} $D \subset \Z^d$ is a map $P : D \to A$.
Each pattern defines a \emph{cylinder set} $[P] = \{ x \in A^{\Z^d} \mid x|_D = P \}$, and these sets form a basis for the profinite topology.
For a pattern $P$ over $A$, denote by $|P|_{\neq 0}$ the total number of nonzero symbols in $P$.
A topologically closed and shift-invariant subset of $A^{\Z^d}$ is called a \emph{shift space}.
If $X \subseteq Y$ are shift spaces, we say $X$ is a \emph{subshift} of $Y$.

Each shift space $X \subseteq A^{\Z^d}$ is defined by a set $F$ of \emph{forbidden patterns}, in the sense that $X = \bigcap_{P \in F} \bigcap_{\vec v \in \Z^d} \sigma^v (A^{\Z^d} \setminus [P])$.
If $F$ can be chosen finite, then $X$ is a \emph{shift of finite type}, or SFT; if $F$ can be chosen computably enumerable, then $X$ is \emph{effectively closed}.
If $X$ is a shift space and $\pi : A \to B$ is any map to another alphabet $B$, the coordinate-wise projection $\pi(X) \subseteq B^{\Z^d}$ is also a shift space.
If $X$ is an SFT, then $\pi(X)$ is a \emph{sofic shift}, and the pair $(X, \pi)$ is its \emph{SFT cover}.
It is well known that all sofic shifts are effectively closed, and that the converse does not hold when $d \geq 2$.

The \emph{language} $\lang(X)$ of a subshift $X \subseteq A^{\Z^d}$ is the set of finite patterns $P$ over $A$ such that $X \cap [P] \neq \emptyset$.
If $X$ is effectively closed, then $\lang(X)$ is a co-computably enumerable set.

A configuration $x \in A^{\Z^2}$ is \emph{vertically constant} if $\sigma^{(0,1)}(x) = x$.
A shift space is vertically constant if all its configurations are.
The \emph{(two-dimensional) lift} of a one-dimensional configuration $x \in A^\Z$ is the vertically constant configuration $x^\dag \in A^{\Z^2}$ with $(x^\dag)_{(i,j)} = x_i$ for all $(i,j) \in \Z^2$.
The lift of a shift space $X \subseteq A^\Z$ is $X^\dag = \{ x^\dag \mid x \in X \}$.
The lift $X^\dag$ is effectively closed if and only if $X$ is.

\begin{theorem}[\cite{DuRoSh12,AuSa13}]
  \label{thm:BlackBox}
  The two-dimensional lift $X^\dag$ of a $\Z$-shift space $X$ is sofic if and only if $X$ is effectively closed.
\end{theorem}

A one-dimensional configuration $x \in A^\Z$ is \emph{ultimately periodic to the right}, if there exists $p > 0$ such that $x_i = x_{i+p}$ holds for all large enough $i > 0$.
Ultimate periodicity to the left is defined symmetrically.

A function $f : \N \to \N$ is \emph{lower semicomputable} if there exists a computable function $g : \N^2 \to \N$ with $g(n,k+1) \geq g(n, k)$ for all $n, k \in \N$ and $f(n) = \max_{k \in \N} g(n,k)$ for all $n \in \N$.
The function $g$ is called a lower approximation of $f$.
Upper semicomputability is defined analogously.

\subsection{Previous work}

We recall the \emph{countable cover conditions} from~\cite{To20}.

\begin{definition}
  Let $X \subseteq A^\Z$ be a one-dimensional shift space.
  We say $X$ satisfies the countable cover conditions if
  \begin{enumerate}
  \item $X$ is effectively closed,
  \item every $x \in X$ is ultimately periodic to the left and right, and
  \item given $u, v, w \in A^+$, it is decidable whether ${}^\infty u v w^\infty \in X$.
  \end{enumerate}
\end{definition}

Note that these conditions imply that $X$ has a computable language.
Namely, since $X$ is effectively closed, $\lang(X)$ is co-computably enumerable.
Conversely, if $P \in \lang(X)$, then $P$ occurs in some $x \in X$ that is ultimately periodic in both directions, and by the third condition we can enumerate such configurations.

\begin{theorem}[\cite{To20}]
  \label{thm:ccc}
  Let $X \subset A^\Z$ be a one-dimensional shift space.
  If the lift $X^\dag \subset A^{\Z^2}$ is a countably covered sofic shift, then $X$ satisfies the countable cover conditions.
\end{theorem}

In~\cite{To20} it was shown that the converse holds if $X$ is a subshift of a one-dimensional countable sofic shift.
Note that a one-dimensional countable sofic shift $Y \subseteq A^\Z$ has a very constrained structure: there is a number $p \geq 1$ such that every configuration $y \in Y$ satisfies $y_i = y_{i+p}$ for all $i \in \Z \setminus I_y$, where $I_y \subset \Z$ has cardinality at most $p$~\cite[Lemma~4.8]{PaSc15}.
In other words, every configuration of $Y$ is locally periodic outside of a finite region (and the local periods and the size of the finite region, though not its diameter, are bounded).

\section{Gap width shifts}

\subsection{Definition and main results}

Recall that the max norm of a vector $\vec v = (v_1, \ldots, v_d) \in \Z^d$ is is $\| v \|_\infty = \max(|v_1|, \ldots, |v_d|)$, and the max distance of $\vec u, \vec v \in \Z^d$ is $\| \vec u - \vec v \|_\infty$.

\begin{definition}
  Let $d \geq 1$, let $A$ be a finite alphabet containing $0$, and let $f : \N \to \N$ be any function.
  The \emph{$d$-dimensional gap width shift} $G_d(A,f) \subset A^{\Z^d}$ is the shift space defined by forbidding those $d$-dimensional patterns $P \in A^D$ such that for some $n \geq 0$ we have $|P|_{\neq 0} > f(n)$, and some $\vec u \neq \vec v \in D$ satisfy $P_{\vec u} \neq 0$, $P_{\vec v} \neq 0$ and $\|\vec u - \vec v \|_\infty \leq n$.
  We denote $G_1(A,f) = G(A,f)$.
  If $A$ is clear from the context, we may also denote $G_d(A,f) = G_d(f)$.
\end{definition}

The gap width shift $G_d(f)$ contains exactly those configurations $x \in A^{\Z^d}$ for which the number of nonzero symbols is at most $f(n)$, where $n$ is the minimal max distance of two nonzero symbols in $x$.
It follows that $G(f)$ only contains configurations with a finite (but usually unbounded) number of nonzero symbols.

We are interested in whether gap width shifts, their lifts or their subshifts are countably covered sofic shifts (or even sofic in the first place).
For the one-dimensional full gap width shifts $G(f)$, soficness of the lift $G(f)^\dag$ corresponds to a computability-theoretic property of the function $f$.
It also gives a necessary condition of the soficness of multidimensional gap width shifts.

\begin{proposition}
  \label{prop:when-sofic}
  Let $f : \N \to \N$ be arbitrary.
  For $d \geq 1$, the shift space $G_d(f)$ is effectively closed if and only if $f$ is upper semicomputable.
  The lift $G_1(f)^\dag$ is sofic if and only if $f$ is upper semicomputable.
\end{proposition}

\begin{proof}
  Suppose $f$ is upper semicomputable with computable approximation $g : \N^2 \to \N$.
  Then $F(f)$ is computably enumerable, since the condition $|P|_{\neq 0} > f(n)$ is equivalent to $\exists k \in \N : |P|_{\neq 0} > g(n, k)$.
  Hence $G(f)$ is effectively closed.
  
  Suppose then that $G(f)$ is effectively closed.
  Let $F$ be the complement of $\lang(G(f))$, which is a computably enumerable set.
  For each $n \in \N$, $f(n)$ is the smallest $k$ such that the $\{0, \ldots, k n\} \times \{0\}^{d-1}$-shaped pattern $a(0^{n-1} a)^k$ is in $F$, where $a \in A_0$ is arbitrary.
  It follows that $f$ is upper semicomputable.
  
  The second claim follows from the first one and Theorem~\ref{thm:BlackBox}.
\end{proof}

As our first main result, we show that the converse of Theorem~\ref{thm:ccc} holds for subshifts of one-dimensional gap width shifts with computable gap functions.
The proof is deferred to Section~\ref{sec:main-constr}.

\begin{theorem}
  \label{thm:main}
  Let $f : \N \to \N$ be a computable function and $X \subseteq G(f)$ a shift space.
  Then $X^\dag$ is a countably covered sofic shift if and only if $X$ satisfies the countable cover conditions.
\end{theorem}

\begin{corollary}
  \label{cor:main}
  The lift $G(f)^\dag$ is a countably covered sofic shift if and only if $f : \N \to \N$ is computable.
\end{corollary}

\begin{proof}
  Suppose that $G(f)^\dag$ is countably covered, so that $G(f)$ satisfies the countable cover conditions.
  Now $f(n) = \min \set{ k \in \N }{ {}^\infty 0 a (0^n a)^k 0^\infty \notin G(f) }$ for each $n \in \N$, where $a \in A_0$ is arbitrary.
  Since ${}^\infty u v w^\infty \in G(f)$ is decidable by assumption, $f$ is computable.
  
  Suppose then that $f$ is computable.
  Then $G(f)$ satisfies the countable cover conditions: it is effectively closed since the set of forbidden patterns $F(f)$ is computably enumerable, every configuration is eventually all-$0$ in both directions, and ${}^\infty u v w^\infty$ can be easily decided by measuring the smallest gap $1 0^n 1$, computing $f(n)$ and comparing that to the number of $1$-symbols.
\end{proof}

In particular, the subclass of one-dimensional gap width shifts with countably covered sofic lifts is also characterized by a computability condition, and is neither empty nor equal to the class of gap width shifts with sofic lifts.

Our second main result concerns two-dimensional gap width shifts.
We show that a large subclass of them are in fact countably covered sofic shifts, including ones that satisfy the conditions of Proposition~\ref{prop:nonsofic-subshift}.
The proof is given in Section~\ref{sec:main2}.

\begin{theorem}
  \label{thm:main2}
  Suppose that $f : \N \to \N$ is nondecreasing and upper semicomputable, and $f(n) < 2^{\sqrt{n}}$ for all large enough $n \in \N$.
  Then $G_2(A,f)$ is a countably covered sofic shift.
\end{theorem}

Interestingly, the condition of the gap function is incomparable to that of Corollary~\ref{cor:main}.
The computational restriction is weaker, intuitively because the lack of a periodic direction allows us to simulate arbitrary infinitely long computations.
On the other hand, we are constrained by a fixed computable upper bound, which stems from the need to transmit information about the number of nonzero symbols in one region of a configuration to another.
We know that the bound is suboptimal, and we have not seriously tried to optimize it since our efforts result in more complex constructions that still cannot reach the bound $2^n$.

\subsection{Gap functions}

The correspondence between the properties of subshifts of $G_d(f)$ and the computability-theoretic properties of the function $f$ is somewhat subtle.
Let us investigate it in more detail.

\begin{definition}
  Let $X \subseteq G_d(f)$ be a shift space.
  Define the \emph{gap function} $f_X : \N \to \N$ of $X$ as follows: $f_X(n)$ is the smallest $k \geq 0$ such that there exists $x \in X$ with $|x|_{\neq 0} = k$ and $\min \{ \| \vec u - \vec v \|_\infty \mid \vec u \neq \vec v, x_{\vec u} \neq 0, x_{\vec v} \neq 0 \} \geq n$.
\end{definition}

In other words, the gap function $f_X : \N \to \N$ is the pointwise smallest nondecreasing function such that $X \subseteq G_d(f_X)$.
For simplicity, we concentrate on the one-dimensional case.

\begin{proposition}
  \label{prop:lower-sc}
  Let $f : \N \to \N$ be any nondecreasing function and suppose $X \subseteq G(f)$ satisfies the countable cover conditions.
  Then $f_X$ is lower semicomputable.
\end{proposition}

\begin{proof}
  As $X$ satisfies the countable cover conditions, we can algorithmically enumerate its configurations in the form ${}^\infty 0 v 0^\infty$ for $v \in A^*$.
  Every such $v$ containing an occurrence of $a 0^{n-1} b$ is a witness for $f_X(n) \geq |v|_{\neq 0}$, and $f_X(n) = |v|_{\neq 0}$ holds for at least one $v$.
\end{proof}

To recap, if $X = G(f)$ is effectively closed, then $f_X = f$ is upper semicomputable by Proposition~\ref{prop:when-sofic}.
This does not hold for proper subshifts of $G(f)$: there exist effectively closed shift spaces $X \subseteq G(f)$ such that $f_X$ is not upper semicomputable.
On the other hand, if $X \subseteq G(f)$ satisfies the countable cover conditions, then $f_X$ is lower semicomputable by Proposition~\ref{prop:lower-sc}.
If we further know that $f_X$ is computable, then the lift $X^\dag$ is a countably covered sofic shift by Theorem~\ref{thm:main}.
But $f_X$ might not be computable (and \emph{a posteriori} not upper semicomputable).
The example below is a witness for both claims.

\begin{example}
  \label{ex:not-computable-gap}
  There exists a subshift of a gap width shift that satisfies the countable cover conditions, but whose gap function is not bounded from above by any computable function.
  Theorem~\ref{thm:main} does not apply in this situation; we do not know whether such shift spaces have countably covered lifts.
  
  Define a shift space $X \subset \{0,1\}^\Z$ as the set of all translates of the following configurations: the all-$0$ configuration; configurations ${}^\infty 0 (1 0^n)^k 1 0^\infty$ for all $n \geq 2$ and $0 \leq k \leq n$; and configurations ${}^\infty 0 1 0^n (1 0^k)^k 1 0^\infty$ for all $n \geq 2$ such that the $n$th Turing machine halts on empty input in exactly $k > n$ steps.
  This set is a shift space.
  Since $f_X$ has Busy Beaver-like growth, it is not bounded from above by a computable function.

  We claim that $X$ satisfies the countable cover conditions.
  First, it is effectively closed: every word with at most two $1$-symbols occurs in $X$; a word of the form $0^a (1 0^n)^k 0^b$ with $k \geq 2$ occurs in $X$ if and only if $2 \leq n$ and $k \leq n$; and all other words that occur in $X$ have the form $0^a 1 0^n (0^k 1)^m 0^b$ for some $a, b \geq 0$ and $2 \leq n < k$ and $2 \leq m \leq k$, and their occurrence can be checked by simulating the $n$th Turing machine for $k$ steps.
  Second, every configuration of $X$ is clearly ultimately periodic.
  Third, whether ${}^\infty u v w^\infty \in X$ can be decided by checking that it is of the correct form, and possibly simulating a Turing machine for a bounded number of steps.
\end{example}

\subsection{The two-dimensional case}

By Proposition~\ref{prop:when-sofic}, $f$ being upper semicomputable is a necessary condition for $G_2(A,f)$ being sofic.
We do not know whether it is sufficient, although this seems plausible for the following reason.
For a number $0 \leq \alpha \leq 2$, let $X_\alpha \subset \{0,1\}^{\Z^2}$ be the shift space obtained by forbidding, for all $n \geq 2$, each $n \times n$ square with more than $n^\alpha$ occurrences of $1$.
In his PhD thesis \cite[Section III]{De21}, Julien Destombes shows that $X_\alpha$ is sofic whenever $0 \leq \alpha < 1$ is an upper semicomputable number, and so is every effectively closed subshift of $X_\alpha$.

The construction uses \emph{fixed-point tilings} introduced in \cite{DuRoSh12}, in which each configuration consists of a regular grid of $N_1 \times N_1$ patterns called \emph{level-1 macrotiles}, the macrotiles are arranged in a grid of $N_1 N_2 \times N_1 N_2$ patterns called \emph{level-2 macrotiles}, and so on.
Each level-$k$ macrotile gathers information about the positions of $1$-symbols inside it, tries to find a forbidden pattern inside it, and if one was not found, passes said information to the level-$(k+1)$ macrotile containing it.
The construction could presumably be modified so that each macrotile keeps track of just the number of $1$-symbols it contains and the smallest max distance between them, and checks whether the constraint of $G_2(A,f)$ is violated.
Such a modification is beyond the scope of this article, as the fixed-point construction is very complex and never produces a countable-to-one cover.

On the other hand, even some seemingly tame two-dimensional gap width shifts have effectively closed subshifts that are not sofic.
Hence, we cannot improve Theorem~\ref{thm:main2} to include arbitrary effectively closed subshifts of the $G_2(A,f)$ without further constraining the growth rate of $f$.
To prove this, we need the following result.

\begin{lemma}[Theorem~2.3 of~\cite{KaMa13}, paraphrased]
  \label{lem:swap}
  For each sofic shift $X \subseteq A^{\Z^2}$ there exists $C > 0$ with the following property.
  Let $m \geq 1$ and let $S \subseteq X$ have cardinality at least $C^m$.
  Then there exist $x \neq y \in S$ and a configuration $z \in X$ such that $z|_{[0, m-1]^2} = x|_{[0, m-1]^2}$ and $z|_{\Z^2 \setminus [0, m-1]^2} = y|_{\Z^2 \setminus [0, m-1]^2}$.
\end{lemma}

The simplest example that violates the conclusion of Lemma~\ref{lem:swap} is the \emph{mirror shift}, a shift space over $\{0,1,2\}$ whose configurations contain a single infinite column of $2$-symbols surrounded by two half-planes over $\{0,1\}$ that are mirror images of each other (or no $2$-symbols).
We implement a version of the mirror shift as a subshift of $G_2(\{0,1,2\}, f)$.

\begin{proposition}
  \label{prop:nonsofic-subshift}
  Suppose that $f : \N \to \N$ is nondecreasing and computable, and that for all $C > 0$ there exists $n \in \N$ with $f(n) > C n$.
  Then there exists a computable subshift of $G_2(\{0,1,2\},f)$ that is not sofic.
\end{proposition}

\begin{proof}
  Given an integer $k \geq 1$, let $n_k \in \N$ be the smallest integer such that $f(k) \geq 1 + 2 \lceil k n_k \log_2 k \rceil$.
  The function $k \mapsto n_k$ is computable since $f$ is.

  Define a subshift $X \subset G_2(\{0,1,2\}, f)$ with the following constraints on a configuration $x \in X$.
  There can be at most one $2$-symbol.
  If $x$ contains at least two nonzero symbols with some distance $k \geq 1$, then within distance $k n_k$, at some position $\vec v$, there must be a $2$-symbol.
  If $x$ contains a $2$-symbol at position $\vec v$, then for some $k \geq 1$, every $1$-symbol is at a position $\vec v + (i k, j k)$ with $i \in [-n_k,-1] \cup [1,n_k]$ and $j \in [0,n_k-1]$.
  Finally, there is a $1$-symbol at $\vec v + (i k, j k)$ if and only if there is one at $\vec v + (-i k, j k)$.
  This subshift is clearly computable.

  We show that the conclusion of Lemma~\ref{lem:swap} does not hold for $X$.
  Let $C \geq 1$ be given.
  Let $T \subseteq \{ (i C, j C) \mid i \in [-n_C,-1] \cup [1,n_C], j \in [0,n_C-1] \}$ have cardinality $\lceil C n_C \log_2 C \rceil$.
  Let $R$ be the set of patterns of domain $[0, C n_C - 1]^2$ that contain $0$s or $1$s at positions of $T$ and $0$s at all other positions.
  Each such pattern $P \in R$ can be extended into a configuration $x^P \in X$ in which $x^P_{(-C,0)} = 2$ by placing a mirror image of $P$ to the left of the $2$-symbol and filling the rest of the configuration with $0$-symbols.
  These extensions are in $G_2(\{0,1,2\}, f)$, since the minimum distance between two nonzero symbols in $x^P$ is at least $C$, and $f(C) \geq 1 + 2 \lceil C n_C \log_2 C \rceil$, which is the number of nonzero symbols in $x^P$.
  
  Let $S = \{ x^P \mid P \in R\}$.
  Then $|S| = |R| \geq 2^{\lceil C n_C \log_2 C \rceil} \geq C^{C n_C}$, but there do not exist $x \neq y \in S$ satisfying the conclusion of Lemma~\ref{lem:swap}, as all such $x$ and $y$ have different contents at $[0, C n_C -1]$, which are mirrored on the left of the $2$-symbol at $(-C,0)$.
\end{proof}

\section{Construction toolbox}

% \begin{lemma}
%   Every finite union of countably covered sofic shifts is a countably covered sofic shift.
% \end{lemma}

In this section we present several auxiliary constructions of countable SFTs that are used to prove Theorems~\ref{thm:main} and~\ref{thm:main2}.
Many of them were first used in some form in~\cite{To20}, but we restate them here both for completeness and for ease of reuse.

Our most fundamental building block is the grid shift, a standard example of a nontrivial two-dimensional SFT.
In our constructions it acts as a way of transferring information from one part of a configuration to everywhere else, as a scaffold on which we can implement computation and various geometric gadgets, and as a ``clock'' that periodically resets said computation.

\begin{example}[Grid shift]
  The grid shift $\grid \subset \agrid^{\Z^2}$ is the SFT defined on the Wang tile set $\agrid$ shown in Figure~\ref{fig:grid}.
  % We denote by $\cgrid$ the set of colors in $\agrid$.
  A configuration of $\grid$ either contains a regular grid of horizontal and vertical lines, or contains no parallel lines.
  In particular, $\grid$ is a countable SFT.
  For each $n \geq 1$, let $\xgrid(n) \in \grid$ be the configuration containing a grid of side length $n$ with a crossing of grid lines at the origin.
\end{example}

\begin{figure}[htp]
  \centering
  \begin{tikzpicture}[scale=0.6]

    \fill [black!20] (0,3) -- (0,4.5) -- (1.5,4.5) -- (1.5,7) -- (4,7);
    \fill [black!20] (0,0) -- (0,0.5) -- (1.5,0.5) -- (1.5,4.5) -- (5.5,4.5) -- (5.5,7) -- (8,7) -- (1,0);
    \fill [black!20] (1.5,0) -- (1.5,0.5) -- (5.5,0.5) -- (5.5,4.5) -- (9.5,4.5) -- (9.5,7) -- (12,7) -- (5,0);
    \fill [black!20] (5.5,0) -- (5.5,0.5) -- (9.5,0.5) -- (9.5,4.5) -- (12,4.5) -- (12,3) -- (9,0);
    \fill [black!20] (9.5,0) -- (9.5,0.5) -- (12,0.5) -- (12,0);

    \draw (0,3) -- (4,7);
    \draw (1,0) -- (8,7);
    \draw (5,0) -- (12,7);
    \draw (9,0) -- (12,3);
    \draw (0,0.5) -- (12,0.5);
    \draw (0,4.5) -- (12,4.5);
    \draw (1.5,0) -- (1.5,7);
    \draw (5.5,0) -- (5.5,7);
    \draw (9.5,0) -- (9.5,7);

    \draw [dotted] (0,0) grid (12,7);

  \end{tikzpicture}
  \caption{A sample configuration of the grid shift $\grid$. Dotted lines are tile borders, along which colors must match.}
  \label{fig:grid}
\end{figure}

The following is essentially Example 14 of~\cite{To20}.
It allows us to constrain the number of nonzero columns in a configuration of $(A^\Z)^\dag$ by the size of a grid on a separate layer.

\begin{lemma}[Counting columns]
  \label{lem:count-columns}
  Fix a finite alphabet $A$ containing $0$.
  There exists an SFT $\cnt \subset (A^\Z)^\dag \times \grid \times \acnt^{\Z^2}$ with the following properties.
  \begin{enumerate}
  \item
    \label{it:fixN}
    For each $n \in \N$, the set $\cnt(n) := \set{ z \in \cnt }{ \pi_{\mathrm{grid}}(z) = \xgrid(n) }$ is countable, and the projection $\pi_A(\cnt(n))$ is exactly the set of configurations of $(A^\Z)^\dag$ with at most $n-2$ nonzero columns.
  \item
    \label{it:fixA}
    For each $x \in (A^\Z)^\dag$ the set $\set{ z \in \cnt }{ \pi_A(z) = x }$ is countable.
  \end{enumerate}
\end{lemma}

Note that we could enforce a bound of $n$ nonzero columns in item~\ref{it:fixN} instead of $n-2$, but this would require a slightly more complex tile set.

\begin{proof}
  The alphabet $\acnt$ consists of two layers of Wang tiles, which are shown in Figure~\ref{fig:count-alph}.
  The first layer is the \emph{counting layer}, and the second one is the \emph{synchronization layer}.
  The alphabet of $\cnt$ also contains an $A$-layer and a grid layer.

  \begin{figure}[htp]
    \centering
    \begin{tikzpicture}
      
      \node at (-0.5,0.75) {(a)};
      
      \begin{scope}[shift={(0,0)}]
        \draw (0,0) rectangle (1,1);
      \end{scope}

      \begin{scope}[shift={(1.5,0)}]
        \draw[pattern=north east lines] (0,0) rectangle (1,1);
      \end{scope}

      \begin{scope}[shift={(3,0)}]
        \fill[pattern=north east lines] (0.5,0) rectangle (1,0.5);
        \draw (0.5,0) -- (0.5,0.5) -- (1,0.5);
        \draw (0,0) rectangle (1,1);
      \end{scope}

      \begin{scope}[shift={(4.5,0)}]
        \fill[pattern=north east lines] (0,0) rectangle (1,0.5);
        \draw (0,0.5) -- (1,0.5);
        \draw (0,0) rectangle (1,1);
      \end{scope}
      
      % \begin{scope}[shift={(4.5,1.5)}]
      %   \fill[black!20] (0,0) rectangle (0.5,0.5);
      %   \fill[black!40] (0.5,0) -- (0.5,0.5) -- (1,0);
      %   \fill[black!20] (1,0) -- (0.5,0.5) -- (0.5,1) -- (1,1);
      %   \draw[dashed] (0,0.5) -- (0.5,0.5) -- (0.5,1);
      %   \draw[very thick] (0.5,0) -- (0.5,0.5);
      %   \draw (0.5,0.5) -- (1,0);
      %   \draw (0,0) rectangle (1,1);
      % \end{scope}
      
      % \begin{scope}[shift={(6,1.5)}]
      %   \fill[black!20] (0,0) rectangle (0.5,0.5);
      %   \fill[black!40] (0.5,0) -- (0.5,0.5) -- (1,0);
      %   \fill[black!20] (1,0) -- (0.5,0.5) -- (1,0.5);
      %   \draw[dashed] (0,0.5) -- (1,0.5);
      %   \draw[very thick] (0.5,0) -- (0.5,0.5);
      %   \draw (0.5,0.5) -- (1,0);
      %   \draw (0,0) rectangle (1,1);
      % \end{scope}
      
      % \begin{scope}[shift={(7.5,1.5)}]
      %   \fill[black!40] (0,0.5) rectangle (1,1);
      %   \draw[very thick] (0,0.5) -- (1,0.5);
      %   \draw (0,0) rectangle (1,1);
      % \end{scope}

      % \draw[fill=black!40] (0,0) rectangle (1,1);

      % \begin{scope}[shift={(1.5,0)}]
      %   \fill[black!40] (0,0) -- (1,0) -- (0,1);
      %   \fill[black!20] (1,1) -- (1,0) -- (0,1);
      %   \draw (1,0) -- (0,1);
      %   \draw (0,0) rectangle (1,1);
      % \end{scope}
      
      \begin{scope}[shift={(6,0)}]
        \fill[pattern=north east lines] (0,0) -- (0,0.5) -- (0.5,0.5) -- (0.5,1) -- (1,1) -- (1,0);
        \draw (0,0.5) -- (0.5,0.5) -- (0.5,1);
        \draw (0,0) rectangle (1,1);
      \end{scope}

      \begin{scope}[shift={(7.5,0)}]
        \fill[pattern=north east lines] (0,0.5) rectangle (1,1);
        \draw[very thick] (0,0.5) -- (1,0.5);
        \draw (0,0) rectangle (1,1);
      \end{scope}
      
      % \begin{scope}[shift={(6,0)}]
      %   \fill[black!20] (0,0) rectangle (1,1);
      %   %   \fill[black!40] (0.5,0) rectangle (1,1);
      %   \draw[very thick] (0.5,0) -- (0.5,1);
      %   \draw (0,0) rectangle (1,1);
      % \end{scope}
      
      % \begin{scope}[shift={(7.5,0)}]
      %   \fill[black!20] (0,0.5) rectangle (1,1);
      %   %   \fill[black!40] (0.5,0.5) rectangle (1,1);
      %   \draw[very thick] (0,0.5) -- (1,0.5);
      %   \draw[very thick] (0.5,0.5) -- (0.5,1);
      %   \draw (0,0) rectangle (1,1);
      % \end{scope}
      
      % \begin{scope}[shift={(9,0)}]
      %   \fill[black!40] (0,0.5) -- (0,1) -- (0.5,0.5);
      %   \fill[black!20] (0.5,0.5) -- (0,1) -- (1,1) -- (1,0.5);
      %   \draw[very thick] (0,0.5) -- (1,0.5);
      %   \draw (0.5,0.5) -- (0,1);
      %   \draw[fill=white] (0.5,0.5) circle (0.07cm);
      %   \draw (0,0) rectangle (1,1);
      % \end{scope}

      % \draw[gray,dashed,rounded corners=3mm] (4.5-1/3,-1/6) -- (4.5-1/3,1+1/4) -- (3-1/6,1+1/4) -- (3-1/6,2.5+1/3) -- (5.5+1/6,2.5+1/3) -- (5.5+1/6,-1/6) -- cycle;
      % \node[draw,fill=white] at (5,-1/3) {1};
      % \draw[gray,dashed,rounded corners=3mm] (6-1/6,-1/6) -- (6-1/6,1+1/4) -- (4.5-1/6,1+1/4) -- (4.5-1/6,2.5+1/6) -- (7+1/6,2.5+1/6) -- (7+1/6,1+1/4) -- (8.5+1/4,1+1/4) -- (8.5+1/5,-1/6) -- cycle;
      % \node[draw,fill=white] at (6.5,-1/3) {2};
      % \draw[gray,dashed,rounded corners=3mm] (7.5-1/6,-1/3) rectangle (10+1/6,2.5+1/6);
      % \node[draw,fill=white] at (9.5,-1/3) {3};
      
      \node at (-0.5,-1.25) {(b)};
      
      \begin{scope}[shift={(0,-2)}]
        \draw (0,0) rectangle (1,1);
      \end{scope}
      
      \begin{scope}[shift={(1.5,-2)}]
        \draw[fill=black!20] (0,0) rectangle (1,1);
      \end{scope}

      \begin{scope}[shift={(3,-2)}]
        \draw[fill=black!40] (0,0) rectangle (1,1);
      \end{scope}
      
      \begin{scope}[shift={(4.5,-2)}]
        \fill[black!20] (0.5,0) rectangle (1,1);
        \draw[very thick] (0.5,0) -- (0.5,1);
        \draw (0,0) rectangle (1,1);
      \end{scope}
      
      \begin{scope}[shift={(6,-2)}]
        \fill[black!40] (0.5,0) rectangle (1,1);
        \draw[very thick] (0.5,0) -- (0.5,1);
        \draw (0,0) rectangle (1,1);
      \end{scope}

      \begin{scope}[shift={(7.5,-2)}]
        \fill[black!20] (0,0) rectangle (0.5,1);
        \draw[dashed] (0.5,0) -- (0.5,1);
        \draw (0,0) rectangle (1,1);
      \end{scope}

      \begin{scope}[shift={(9,-2)}]
        \fill[black!20] (0,0) rectangle (1,0.5);
        \fill[black!40] (0,0.5) rectangle (1,1);
        \draw (0,0.5) -- (1,0.5);
        \draw (0,0) rectangle (1,1);
      \end{scope}

      \begin{scope}[shift={(0,-3.5)}]
        \fill[black!20] (1,1) -- (1,0) -- (0,1);
        \fill[black!40] (0,0) -- (1,0) -- (0,1);
        \draw (0,1) -- (1,0);
        \draw (0,0) rectangle (1,1);
      \end{scope}

      \begin{scope}[shift={(1.5,-3.5)}]
        \fill[black!20] (0.5,1) -- (0.5,0.5) -- (1,0) -- (1,1);
        \fill[black!40] (0.5,0) -- (0.5,0.5) -- (1,0);
        \draw (0.5,0.5) -- (1,0);
        \draw [very thick] (0.5,0) -- (0.5,1);
        \draw [fill=white] (0.5,0.5) circle (0.07cm);
        \draw (0,0) rectangle (1,1);
      \end{scope}

      \begin{scope}[shift={(3,-3.5)}]
        \fill[black!20] (0.5,0) rectangle (1,0.5);
        \fill[black!40] (0.5,0.5) rectangle (1,1);
        \draw (0.5,0.5) -- (1,0.5);
        \draw [very thick] (0.5,0) -- (0.5,1);
        \draw (0,0) rectangle (1,1);
      \end{scope}

      \begin{scope}[shift={(4.5,-3.5)}]
        \fill[black!20] (0,0) rectangle (0.5,0.5);
        \fill[black!20] (0,1) -- (0.5,0.5) -- (0.5,1);
        \fill[black!40] (0,0.5) -- (0.5,0.5) -- (0,1);
        \draw (0,0.5) -- (0.5,0.5) -- (0,1);
        \draw [dashed] (0.5,0) -- (0.5,1);
        \draw (0,0) rectangle (1,1);
      \end{scope}
      
    \end{tikzpicture}
    \caption{The two layers of the alphabet $\acnt$ of the SFT $\cnt$: (a) the counting layer and (b) the synchronization layer.}
    \label{fig:count-alph}
  \end{figure}

  The counting layer forms a unary counter in each vertical segment between two horizontal grid lines.
  This counter is incremented whenever it crosses a nonzero column on the $A$-layer.
  It cannot increase beyond the grid lines, which restricts the number of nonzero columns exactly as desired.
  The counter may have a nonzero initial value, so if we could choose it independently between any two horizontal grid lines, our SFT would be uncountable.
  Thus, the synchronization layer is used to force each counter to have the same initial value.

  Each layer must respect the local matching rules of its Wang tiles.
  In addition, the layers are constrained by the following rules, where $c \in A \times \agrid \times \acnt$ is an arbitrary symbol:
  \begin{enumerate}
  \item
    \label{rule:1}
    The grid layer of $c$ contains a horizontal line if and only if its counting layer contains a thick horizontal line (with a shaded region above it).
  \item
    \label{rule:2}
    The grid layer of $c$ contains a vertical line if and only if the synchronization layer contains a solid vertical line.
  \item
    \label{rule:4}
    If the synchronization layer of $c$ contains a horizontal line, then the grid layer contains a horizontal line.
  \item
    \label{rule:3}
    If the synchronization layer of $c$ contains a white dot, then the counting layer of $c$ has a thin horizontal line on its west border (with a shaded region below it).
  \item
    \label{rule:5}
    If the counting layer of $c$ has a thin horizontal line on its west border, then the line turns north if and only if the $A$-layer is nonzero.
  \end{enumerate}

  \begin{figure}[htp]
    \centering

    \begin{tikzpicture}[scale=0.6]

      \pgfmathsetmacro{\ymax}{12}
      \pgfmathsetmacro{\xmax}{19}

      \begin{scope}
        \clip (0,0) rectangle (\xmax, \ymax);
        \foreach \xl/\xr in {0/1,5/7,10/12,15/18}{
          \fill [black!20] (\xl+0.5, 0) rectangle (\xr+0.5, \ymax);
        }
        \foreach \x in {1,7,12,18}{
          \draw [dashed] (\x+0.5, 0) -- (\x+0.5, \ymax);
        }
        \foreach \y in {0,5,10}{
          \draw (0, \y+1.5) -| (4.5, \y+2.5) -| (10.5, \y+3.5) -- (\xmax, \y+3.5);
          \foreach \x/\h in {0/1,5/2,10/2,15/3}{
            \fill [black!40] (\x+0.5, \y+0.5) -- ++(0,\h) -- ++(\h,-\h);
            % \draw [very thick] (\x+0.5, \y+0.5) -- ++(0,\h);
            \draw (\x+0.5, \y+\h+0.5) -- ++(\h,-\h);
          }
          \fill [pattern=north east lines,opacity=0.5] (0,\y+0.5) -- (0, \y+1.5) -| (4.5, \y+2.5) -| (10.5, \y+3.5) -- (\xmax, \y+3.5) |- cycle; 
          \draw [very thick] (0, \y+0.5) -- (\xmax, \y+0.5);
        }
        \foreach \x/\dy in {0/1,5/2,10/2,15/3}{
          \draw [very thick] (\x+0.5, 0) -- ++(0,\xmax);
          \foreach \y in {0,5,10}{
            \draw [fill=white] (\x+0.5, \y+\dy+0.5) circle (0.12cm);
          }
        }
      \end{scope}

      \draw [dotted] (0,0) grid (\xmax, \ymax);
      \node at (4.5,\ymax+0.5) {$\downarrow$};
      \node at (10.5,\ymax+0.5) {$\downarrow$};
      
    \end{tikzpicture}
    
    \caption{A sample configuration of $\cnt$. Depicted are the grid lines of the grid layer, the counting layer, and the synchronization layer. The two arrows mark nonzero columns of the $A$-layer.}
    \label{fig:count-conf}
  \end{figure}
  
  We describe the structure of configurations $z \in \cnt$; see Figure~\ref{fig:count-conf} for reference.
  Suppose that the grid layer of $z$ contains a grid of finite side length $n$, that is, some translate of $z$ is in $\cnt(n)$.
  By rules~\ref{rule:1} and~\ref{rule:2}, this grid completely determines the positions of all thick horizontal lines of the counting layer and all solid vertical lines of the synchronization layer.
  By rule~\ref{rule:4} it also determines the positions of horizontal lines of the synchronization layer, which must coincide with those of the other layers.
  The bi-infinite stripe between each pair of thick horizontal lines of the counting layer contains a single bi-infinite thin path.
  By rule~\ref{rule:5} the path goes up by one tile whenever it meets a nonzero column on the $A$-layer, and continues horizontally otherwise.
  Since there are $n-1$ rows between two adjacent grid lines, the $A$-layer of $z$ can contain at most $n-2$ nonzero columns.
  
  The positions of the thin paths of different stripes in $z$ are synchronized by the synchronization layer.
  Between each pair of vertical grid lines lies a single dashed vertical line of the synchronization layer.
  Wherever it crosses a horizontal grid line, a diagonal line is sent to the northwest.
  By rule~\ref{rule:3}, it must meet the path of the counting layer and a vertical grid line at the same position (at a white dot).
  Thus, the heights of the counters are synchronized at each vertical grid line.
  Since there are also finitely many choices for the $A$-layer, the set $\cnt(n)$ is countable.
  We also see that the $A$-layer can be chosen freely, as long as it has at most $n-2$ nonzero columns.
  This proves item~\ref{it:fixN}.
  
  For item~\ref{it:fixA}, fix a configuration $x \in (A^\Z)^\dag$, and let $z \in \cnt$ be such that $\pi_A(z) = x$.
  We show that there are a countable number of choices for $z$.
  If the grid layer of $z$ contains a grid of finite side length, the claim follows from item~\ref{it:fixN}.
  
  Suppose it does not contain such a grid.
  Then the counting layer contains at most one thick horizontal line, and hence at most two thin paths.
  Since the $A$-layer is fixed, the position at which a thin path crosses the y-axis determines the entire path, so the number of choices for the counting layer is countable.
  The synchronization layer contains at most one solid vertical line (and hence at most two dashed vertical lines), and its horizontal line segments (the south sides of the dark gray triangles) are constrained to a single line, which is the sole horizontal line of the grid layer if one exists.
  There cannot be more than two infinite diagonal lines, and the positions of finite diagonal segments are determined by the other lines.
  Thus, we have countably many choices for the synchronization layer as well.
\end{proof}

We now introduce a method for simulating computation in countable SFTs.
To our best knowledge, its first use is in~\cite{SaTo13}.

\begin{example}[Counter machines]
  \label{ex:counter-machines}
  We describe a way of simulating computation by nondeterministic oracle counter machines in SFTs.
  We do not explicitly define a $k$-counter machine, but it consists of a finite oracle alphabet $A$, a finite state set $Q$, an initial state $q_0 \in Q$, and a transition relation $\delta$.
  The machine has $k$ counters that store one natural number each, which are initialized to 0, and possibly an infinite oracle tape $x \in A^\N$.
  If the oracle tape exists, then one of the counters is designated as the \emph{oracle counter}.
  Alternatively, the oracle tape may be a bi-directional configuration $x \in A^\Z$, in which case each counter will store an integer instead of a natural number.
  In this latter case the counter machine is called \emph{bidirectional}.
  
  In one computation step the machine will check the relative order of its counters, which of them have value 0, as well as the symbol $x_n \in A$, where $n$ is the current value of the oracle counter.
  Based on the above data and the current state, it will change the value of any of the counters by at most 1 and assume a new state.
  The validity of a given transition is determined by the relation $\delta$.
  If there are two or more valid actions, one is chosen nondeterministically.
  
  The machine may also have a designated final state $q_f$.
  If the machine enters $q_f$, it \emph{accepts}, and if it ever has no possible next state, it \emph{rejects}.
  
  Let $M$ be a one-directional $k$-counter machine as above.
  Define $H_M = H'_M \times \{{\leftarrow},{\rightarrow}\}$ and $A_M = \{\#\} \cup ((H_M \cup \{L,R\}) \times \{Z,P\}^k \times A)$, where $H'_M$ is an auxiliary alphabet that contains a transition from $\delta$, a partial ordering of the $k$ counters, and a symbol of $A$.
  Define an SFT $X_M \subset A_M^{\Z^2}$ as follows.
  First, for each position $\vec v \in \Z^2$, there is a non-$\#$ symbol at $\vec v$ if and only if $\vec v + (0,1)$ and $\vec v + (1,1)$ both contain non-$\#$ symbols.
  Then each connected component of non-$\#$ cells is shaped like a translate of $\{(i,j) \mid j \geq 0, 0 \leq i \leq j\}$.
  Such a component is called a \emph{computation cone}, and its southmost cell is called its \emph{vertex}.
  
  We interpret time as increasing upward and introduce (nondeterministic) local rules that transform a horizontal row of a computation cone into the one above it.
  Each row will contain exactly one symbol $(q,d) \in H_M$ on the first layer, interpreted as a ``zig-zag head'' traveling in the direction $d \in \{{\leftarrow},{\rightarrow}\}$.
  The symbols $L$ and $R$ fill the rest of the row on this layer to the left and right of the head, respectively.
  The zig-zag head bounces between the east and west borders of the computation cone, and one such back-and-forth sweep corresponds to a single computation step of $M$.
  Each counter is represented by the number of $P$-symbols on its own layer, which has the form $P^m Z^n$ for some $m, n \geq 0$.
  As the zig-zag head travels the length of the computation cone, it records the relative order of the counter values in its internal state, and updates counter values based on the transition it contains.

  Let us go into a little more detail.
  Each row of a computation cone consists of $k+2$ layers.
  The first layer must be in the language $L^* H_M R^*$ (storing the position and internal state of the zig-zag head) and the next $k$ layers must be in $P^* Z^*$ (storing the value $m \geq 0$ of one counter as $P^m Z^n$).
  The final layer is in $A^*$, and this layer of each row of the cone is a prefix of the one above; their limit corresponds to the oracle tape.
  At the vertex of the cone, the head contains an arbitrary transition from the initial state $q_0$ to some other state.
  
  The head travels in the direction of its arrow component, with speed 2 to the east and speed 1 to the west.
  The simulation of one computation step of $M$ begins at the west border of the cone, where the zig-zag head picks a new transition $t \in \delta$ into its $H'_M$-component and walks east.
  When it steps east over a $PZ$-pattern on one of the counter layers, it updates that counter's value if specified to do so by the transition $t$.
  Upon reaching the east border of the cone, it turns west and resets the partial order in its state into an emtpy one.
  When the head steps west over a $PZ$-pattern, it records into its internal state the place of the counter in the ordering of all counters, as well as the symbol on the $A$-layer in case of the oracle counter.
  Upon returning to the west border, the head picks a new valid transition from $\delta$ (based on the $A$-symbol and the now total order it stores) and begins the next computation step.
  % In this way, one back-and-forth sweep corresponds to one computation step of $M$.
  
  If the machine is bidirectional, we modify the computation cone to grow both to the left and to the right.
  The left and right halves of the cone have different background colors, and the central column of the cone corresponds to counter value $0$.
  
  The most important property of $X_M$ is that if we ignore the $A$-component, the set of ``degenerate'' configurations -- those that do not contain a simulated computation of $M$ -- is countable.
  Indeed, such a configuration has at most one connected region of non-$\#$ cells, which is shaped like a half-plane and contains at most one infinite sweep of the zig-zag head.
  We will modify this basic construction as needed, in particular by allowing the computation cones to be cut short when $M$ halts, and attaching them to grid cells.
  
  Recall from~\cite[Section~11.2]{Mi67} that any Turing machine computation can be simulated by a counter machine computation in a way that incurs an exponential blowup in the number of computation steps.
  In our construction, the first $n$ simulated steps of a counter machine computation fit in an exponential-size square pattern.
  Hence, the blowup from a Turing machine computation to a pattern is doubly exponential.
\end{example}

\begin{example}[Simulating tile sets in grids]
  \label{ex:counter-machine-grids}
  Let $(T_n)_{n \in \N}$ be a computable sequence of Wang tile sets whose colors are in $\N$, meaning that each $T_n$ is a finite subset of $\N^4$.
  We decribe a way of simulating each $T_n$ by a single SFT, which is countable if each $T_n$ is.
  
  Take the grid shift $\grid$ and superimpose on each gray region a computation of a counter machine using a variation of the SFT $X_M$ in which the computation cones are allowed to be cut off at a horizontal border of the grid provided that the simulated machine has halted, and the oracle tape is blank.
  We give $M$ one counter that corresponds to a choice of which tile set $T_n$ to use, and four counters that correspond to the four sides of a tile.
  The machine $M$ first nondeterministically sets each of these counters to some value (by incrementing them some number of times) and never modifies them again.
  Then it simply checks whether the 5-tuple $(n, t_e, t_n, t_w, t_s)$ they form satisfies $(t_e, t_n, t_w, t_s) \in T_n$, halting if it does.
  In this case we say that the machine, or the grid cell containing it, simulates the tile $(t_e, t_n, t_w, t_s)$ of $T_n$.
  
  We also use signals to transfer information between the machines in adjacent grid cells.
  We enforce each simulated machine in a single configuration to choose the same value of $n$, and for the direction counters to be related as in a Wang tiling (so the north counter of a grid cell must equal the south counter of its north neighbor, and analogously for west and east).
  In this way, the grid cells of a valid configuration of the SFT together simulate a configuration of $T_n$.

  The width of the grid is not constrained, except that if a simulated machine does not have enough space to finish its computation, a tiling error results.
  However, for each $n$ there exists $k \geq 1$ such that any grid of width at least $k$ can simulate an arbitrary configuration of $T_n$.
  
  See~\cite[Theorem~6.2]{SaTo13} for a more detailed construction and analysis.
\end{example}

Consider the following scenario.
We have a vertically constant layer over $A$ that we wish to constrain.
The constraints we wish to apply depend on some integer value $n \geq 1$, which in our case is the minimum distance between two nonzero symbols.
We can add a grid layer, somehow extract $n$ from the $A$-layer of a nondegenerate configuration and constrain the grid cells to have width $n$.
Also, from $n$ we can compute a Wang tile set $T_n$ that, when coupled with the $A$-layer in a suitable way, implements the constraints we want.
We would like to add a layer given by Example~\ref{ex:counter-machine-grids} that simulates all the tile sets $T_n$, and couple it with both the grid layer (where it gets the value $n$) and the $A$-layer (which it constrains).
The issue is that the simulated configurations of $T_n$ are not on the same ``scale'' as the $A$-layer: they are stretched over grids whose width we cannot control, so that each simulated tile of $T_n$ contains an $m \times m$ pattern of $A$-symbols, for some unknown and typically very large $m \geq 1$.

In the next example, which a new contribution, we provide an interface through which the simulated tilings of $T_n$ can still access and constrain arbitrary columns of the $A$-layer.
The tile sets $T_n$ have to be specifically designed to make use of this interface; in the above scenario, we would probably have to redesign them completely.

\begin{example}[Coupling simulated tile sets with vertical layers]
  \label{ex:probe}
  We describe a modification to the SFT $X$ of Example~\ref{ex:counter-machine-grids} that adds a vertically constant layer over a separate alphabet $A$, called the \emph{base alphabet layer}, and gives the simulated tile sets access to it.
  For this, we assume that each tile set $T_n$ has a separate layer, called the \emph{probe interface layer}, over the auxiliary alphabet $B = \{0,1,2\} \cup (A \times \{0,1\})$.
  All $2 \times 1$ patterns except $00$, $0a$, $a1$, $a2$, $11$, $12$ and $22$ for $a \in A \times \{0,1\}$ are forbidden on the probe interface layer.
  Hence, each row of this layer has the form ${}^\infty 0 a 1^k 2^\infty$ for some $a \in A \times \{0,1\}$ and $k \geq 0$ (or a degenerate version).
  As the name suggests, the probe interface layer is used to ``probe'' the contents of the base alphabet layer.

  The interface between the base alphabet layer and the simulated tiling of $T_n$ is the following.
  As in Example~\ref{ex:counter-machine-grids}, we have a grid layer, and each grid cell simulates a tile of $T_n$ for a common $n$.
  Suppose that on some row of the simulated tiling, simulated by $k+2$ consecutive grid cells, we have a word of the form $w = (a,c) 1^k 2$ on the probe interface layer, with $k \geq 0$ and $(a,c) \in A \times \{0,1\}$.
  Then $a$ must be equal to the symbol of the base alphabet layer on the $k$th column of the grid cell that simulates $(a,c)$.
  We think of this as the simulated tile set $T_n$ placing the word $w$ at a certain position of its probe interface layer in order to access a certain column of the non-simulated base alphabet layer.
  Also, $c = 1$ if and only if the probed column is the rightmost one of the grid cell; this lets the tile set $T_n$ avoid attempting to probe beyond the grid cell that simulates $(a,c)$, which our construction does not allow.
  Note that this is not an essential restriction, since different rows of $T_n$ may probe columns of different grid cells, so that a single simulated configuration of $T_n$ may access every column of the base alphabet layer.

  We construct the counter machine $M$ of Example~\ref{ex:counter-machine-grids} so that it contains a special \emph{probe state} $q_b$ for each $b \in B$.
  It will enter a probe state $q_b$ at some point during its computation if and only if the probe layer of the tile it simulates equals $b$.
  Hence, during a finite computation $M$ will enter exactly one of the probe states before halting.
  
  We superimpose on $X \times (A^\Z)^\dag$ a configuration over the tile set $\aprb$ of Figure~\ref{fig:probe-alphabet}, called the \emph{probe implementation layer}, and further constrain it using the following rules:
  \begin{enumerate}
  \item
    \label{it:prb1}
    The thick horizontal and vertical lines of the probe implementation layer must coincide with those of the grid layer of $X$.
  \item
    \label{it:prb2}
    The highlighted tiles in Figure~\ref{fig:probe-alphabet}, which include only those with $b_0 \in A \times \{0\}$, are called the \emph{probe tiles}.
    The $A$-part of the symbol $b_0$, $a_0$ or $a_1$ of a probe tile must be equal to the symbol on the $A$-layer.
  \item
    \label{it:prb3}
    If the west border of a grid cell has on its right the zig-zag head of a simulated machine $M$ in a probe state $q_b$ for $b \in B$, then the corresponding vertical line of the probe layer must hold the symbol $b$.
  \end{enumerate}
  Denote by $H \subseteq X \times (A^\Z)^\dag \times \aprb^{\Z^2}$ the SFT thus obtained.
  By rule~\ref{it:prb1} we can think of the probe implementation layer as decorations over the grid that simulates copies of $M$.
  The probe implementation layer of a sample configuration is depicted in Figure~\ref{fig:probe-conf}, which the reader should consult to better understand the following explanation.

  \begin{figure}[htp]
    \centering

    \begin{tikzpicture}

      \draw[gray,dashed,rounded corners=3mm] (8.75, 5.5) -| (10.25, 1.75) -| (7.15, 3.5) -- (8.75, 3.5) -- cycle;
      
      \begin{scope}[shift={(0,4)}]
        \draw (0,0) rectangle (1,1);
      \end{scope}

      \begin{scope}[shift={(1.5,4)}]
        \draw[fill=black!20] (0,0) rectangle (1,1);
      \end{scope}

      \begin{scope}[shift={(3,4)}]
        \draw[fill=black!40] (0,0) rectangle (1,1);
      \end{scope}

      \begin{scope}[shift={(4.5,4)}]
        \fill[black!20] (0.5,0.5) rectangle (0,1);
        \draw (0,0) rectangle (1,1);
        \draw[very thick] (0.5,0) -- (0.5,1);
        \draw[very thick] (0,0.5) -- (1,0.5);
        \fill (0,0.4) arc (-90:90:0.1cm);
        \node at (0.5,-0.25) {$b$};
        \node at (0.5,1.25) {$1$};
      \end{scope}

      \begin{scope}[shift={(6,4)}]
        \fill[black!20] (1,0.5) rectangle (0,1);
        \draw (0,0) rectangle (1,1);
        \draw[very thick] (0.5,0) -- (0.5,1);
        \draw[very thick] (0,0.5) -- (1,0.5);
        \fill (0,0.4) arc (-90:90:0.1cm);
        \node at (0.5,-0.25) {$b$};
        \node at (0.5,1.25) {$1$};
      \end{scope}

      \begin{scope}[shift={(7.5,4)}]
        \fill[black!40] (0.5,0.5) rectangle (1,1);
        \draw (0,0) rectangle (1,1);
        \draw[very thick] (0.5,0) -- (0.5,1);
        \draw[very thick] (0,0.5) -- (1,0.5);
        % \fill (0,0.4) arc (-90:90:0.1cm);
        \node at (0.5,-0.25) {$b$};
        \node at (0.5,1.25) {$a_*$};
      \end{scope}

      \begin{scope}[shift={(9,4)}]
        \draw (0,0) rectangle (1,1);
        \draw[very thick] (0.5,0) -- (0.5,1);
        \draw[very thick] (0,0.5) -- (1,0.5);
        % \fill (0,0.4) arc (-90:90:0.1cm);
        \node at (0.5,-0.25) {$b$};
        \node at (0.5,1.25) {$b_0$};
      \end{scope}

      % \begin{scope}[shift={(10.5,4)}]
      %   \draw (0,0) rectangle (1,1);
      %   \draw[very thick] (0.5,0) -- (0.5,1);
      %   \draw[very thick] (0,0.5) -- (1,0.5);
      %   \draw[fill=white] (0.5,0.5) circle (0.06cm);
      %   \node at (0.5,-0.25) {$b$};
      %   \node at (0.5,1.25) {$a_0$};
      % \end{scope}

      \begin{scope}[shift={(0,2)}]
        \draw (0,0) rectangle (1,1);
        \draw[very thick] (0,0.5) -- (1,0.5);
      \end{scope}

      \begin{scope}[shift={(1.5,2)}]
        \fill[black!20] (0,0.5) rectangle (1,1);
        \draw (0,0) rectangle (1,1);
        \draw[very thick] (0,0.5) -- (1,0.5);
      \end{scope}

      \begin{scope}[shift={(3,2)}]
        \fill[black!20] (0,0.5) rectangle (1,1);
        \draw (0,0) rectangle (1,1);
        \fill (1,0.6) arc (90:270:0.1cm);
        \draw[very thick] (0,0.5) -- (1,0.5);
      \end{scope}

      \begin{scope}[shift={(4.5,2)}]
        \fill[black!40] (0,0.5) rectangle (1,1);
        \draw (0,0) rectangle (1,1);
        \draw[very thick] (0,0.5) -- (1,0.5);
      \end{scope}

      \begin{scope}[shift={(6,2)}]
        \fill[black!20] (0,0) rectangle (1,0.5);
        \draw (0,0) rectangle (1,1);
        \draw[dashed] (0,0.5) -- (1,0.5);
      \end{scope}

      \begin{scope}[shift={(7.5,2)}]
        \fill[black!20] (0,1) -- (0.5,0.5) -- (1,0.5) -- (1,1);
        \fill[black!40] (0,0.5) -- (0.5,0.5) -- (0,1);
        \draw (0,0) rectangle (1,1);
        \draw[very thick] (0,0.5) -- (1,0.5);
        \draw (0,1) -- (0.5,0.5);
        % \draw[fill=white] (0.5,0.5) circle (0.06cm);
        \node at (-0.05,1.25) {$a_0$};
      \end{scope}
      
      \begin{scope}[shift={(9,2)}]
        \fill[black!20] (0,1) -- (0.5,0.5) -- (1,0.5) -- (1,1);
        \fill[black!40] (0,0.5) -- (0.5,0.5) -- (0,1);
        \draw (0,0) rectangle (1,1);
        \draw[very thick] (0,0.5) -- (1,0.5);
        \fill (1,0.6) arc (90:270:0.1cm);
        \draw (0,1) -- (0.5,0.5);
        % \draw[fill=white] (0.5,0.5) circle (0.06cm);
        \node at (-0.05,1.25) {$a_1$};
      \end{scope}

      \begin{scope}[shift={(0,0)}]
        \draw (0,0) rectangle (1,1);
        \draw[very thick] (0.5,0) -- (0.5,1);
        \node at (0.5,-0.25) {$b$};
        \node at (0.5,1.25) {$b$};
      \end{scope}

      \begin{scope}[shift={(1.5,0)}]
        \draw[fill=black!20] (0,0) rectangle (1,1);
        \draw[very thick] (0.5,0) -- (0.5,1);
        \node at (0.5,-0.25) {$1$};
        \node at (0.5,1.25) {$1$};
      \end{scope}

      \begin{scope}[shift={(3,0)}]
        \fill[black!40] (0.5,0) rectangle (1,1);
        \draw (0,0) rectangle (1,1);
        \draw[very thick] (0.5,0) -- (0.5,1);
        \node at (0.5,-0.25) {$a_*$};
        \node at (0.5,1.25) {$a_*$};
      \end{scope}

      \begin{scope}[shift={(4.5,0)}]
        \fill[black!20] (0,0) rectangle (0.5,0.5);
        \draw (0,0) rectangle (1,1);
        \draw[very thick] (0.5,0) -- (0.5,1);
        \draw[dashed] (0,0.5) -- (0.5,0.5);
        \node at (0.5,-0.25) {$1$};
        \node at (0.5,1.25) {$1$};
      \end{scope}

      \begin{scope}[shift={(6,0)}]
        \fill[black!20] (0,0) -| (1,0.5) -| (0.5,1) -| cycle;
        \draw (0,0) rectangle (1,1);
        \draw[very thick] (0.5,0) -- (0.5,1);
        \draw[dashed] (0.5,0.5) -- (1,0.5);
        \node at (0.5,-0.25) {$1$};
        \node at (0.5,1.25) {$1$};
      \end{scope}

      \begin{scope}[shift={(7.5,0)}]
        \fill[black!20] (0.5,0.5) -- (1,0.5) -- (1,0);
        \fill[black!40] (0.5,0.5) -- (0.5,0) -- (1,0);
        \draw (0,0) rectangle (1,1);
        \draw[very thick] (0.5,0) -- (0.5,1);
        \draw[dashed] (0.5,0.5) -- (1,0.5);
        \draw (0.5,0.5) -- (1,0);
        \node at (0.5,-0.25) {$a_*$};
        \node at (0.5,1.25) {$a_*$};
        \node at (1.1,-0.25) {$a_*$};
      \end{scope}

      \begin{scope}[shift={(9,0)}]
        \fill[black!20] (0,1) -- (1,0) -- (1,1);
        \fill[black!40] (0,1) -- (1,0) -- (0,0);
        \draw (0,0) rectangle (1,1);
        \draw (0,1) -- (1,0);
        \node at (-0.1,1.25) {$a_*$};
        \node at (1.1,-0.25) {$a_*$};
      \end{scope}
      
    \end{tikzpicture}
    
    \caption{The tile set $\aprb$ of the probe implementation layer. Here $b$ ranges over $B$, $b_0$ ranges over $\{0,1,2\} \cup (A \times \{0\}) \subset B$, $a_*$ ranges over $A \times \{0,1\}$, and $a_i$ for $i \in \{0,1\}$ ranges over $A \times \{i\}$.}
    \label{fig:probe-alphabet}
  \end{figure}

  \begin{figure}[htp]
  \centering

  \begin{tikzpicture}[scale=0.6]

    \pgfmathsetmacro{\ymax}{13}
    \pgfmathsetmacro{\xmax}{19}

    \fill [black!20] (5.5,2.5) -| ++(4,-1) -| ++(4,-1) -- (7.5,0.5) -- cycle;
    \fill [black!20] (1.5,7.5) -| ++(4,-1) -| ++(4,-1) -| ++(4,-1) -- (4.5,4.5) -- cycle;
    \fill [black!40] (5.5,2.5) |- ++(2,-2) -- cycle;
    \fill [black!40] (1.5,7.5) |- ++(3,-3) -- cycle;
    
    \draw[dotted] (0,0) grid (\xmax,\ymax);

    \foreach \y in {0,4,8,12}{
      \draw [very thick] (0,\y+0.5) -- (\xmax,\y+0.5);
    }
    \foreach \x in {1,5,9,13,17}{
      \draw [very thick] (\x+0.5,0) -- (\x+0.5,\ymax);
    }

    \draw (5.5,2.5) -- ++(2,-2);
    \draw (1.5,7.5) -- ++(3,-3);

    \foreach \x/\y in {5/2,9/1, 1/7,5/6,9/5}{
      \draw [dashed] (\x+0.5,\y+0.5) -- ++(4,0);
    }
    \foreach \x/\y in {9/0,13/0, 5/4,9/4,13/4}{
      \fill (\x,\y+0.5) circle (0.15cm);
    }

    \foreach \x/\y/\l in {
      1/3/0, 5/3/a_0, 9/3/1, 13/3/1, 17/3/2,
      1/5/a_1, 5/5/1, 9/7/1, 13/7/1, 17/7/2,
      1/11/0,5/11/a_0,9/11/2,13/11/2,17/11/2
    }{
      \node [fill=white,minimum width=0.35cm,minimum height=0.45cm,inner sep=0cm] at (\x+0.5,\y+0.5) {$\l$};
    }
    
  \end{tikzpicture}
  
  \caption{A configuration over $\aprb$, superimposed over a configuration of $\grid$. Each vertical segment is labeled with the type of symbol it carries (each tile of the segment carries the same symbol); $a_i$ stands for $A \times \{i\}$.}
  \label{fig:probe-conf}
\end{figure}

  Consider a configuration $(x, y, z) \in H$, where $x \in X$ contains a grid of width $m \geq 1$ that simulates a configuration of tile set $T_n$ for some $n \in \N$, and $y \in (A^\Z)^\dag$ is the base alphabet layer.
  By rule~\ref{it:prb1}, the horizontal and vertical lines of the probe implementation layer $z$ are aligned with the grid.
  Each grid cell $C$ contains a simulated copy of the machine $M$, which enters exactly one probe state $q_b$ during its computation.
  Here, $b \in B$ is the probe interface layer of the tile simulated by $C$.
  At the point where the zig-zag head of $C$ touches the west border of its computation cone in state $q_b$, rule~\ref{it:prb3} causes the neighboring tile on the west border of $C$, and hence the entire west border, to be colored with the symbol $b$.
  
  By our assumption on $T_n$, the probe interface layer of each row of the simulated tiling contains at most one symbol from $A \times \{0,1\}$, and on its right some number of $1$-symbols, all other symbols coming from $\{0,2\}$.
  Thus, a given row of the grid contains at most one grid cell $C_0$ whose west border holds a value $(a, c) \in A \times \{0,1\}$.
  Cell $C_0$ is the one whose $A$-layer we are probing on this particular row.
  Thus, if the next $k \geq 0$ grid cells $C_1, \ldots, C_k$ to the east of $C_0$ hold $1$-symbols on their west borders, we need to access the base alphabet layer of the $k$th column of $C_0$ from the left.
  We describe how this is achieved with the help of the probe implementation layer.

  Suppose first that $k \geq 1$.
  Note from Figure~\ref{fig:probe-alphabet} that in a crossing of grid lines (last 4 tiles on row 1), the north border of the tile holds a $1$-symbol if and only if the northwest part of the tile is light gray.
  Hence, a grid cell contains a gray region if and only if its east border holds a $1$-symbol.
  The gray region of each grid cell that contains one is a rectangle delimited by its south, west and east borders and a dashed horizontal line.
  If two neighboring grid cells contain gray regions, their heights differ by 1, the west one being higher (see tiles 4 and 5 on row 3 of Figure~\ref{fig:probe-alphabet}).
  
  The east border of grid cell $C_k$ does not hold a $1$-symbol, so $C_k$ does not contain a gray region.
  Inductively, for each $i = 0, \ldots, k-1$ the grid cell $C_i$ has a gray region of height $k-i$.
  From the northwest corner of the gray region of $C_0$, which has height $k$, a diagonal signal is sent to the southeast (tile 6 on row 3 of Figure~\ref{fig:probe-alphabet}).
  
  The diagonal signal carries the same symbol $(a, c)$ as the west border of $C_0$ (tiles 6 and 7 on row 3 of Figure~\ref{fig:probe-alphabet}).
  As it hits the south border of $C_0$ at distance $k$ from its southwest corner (tiles 6 and 7 on row 2 of Figure~\ref{fig:probe-alphabet}), rule~\ref{it:prb2} guarantees that the symbol on the base alphabet layer at that position is $a$.
  Also, since the diagonal signal cannot intersect another vertical grid line, we must have $k < m$.
  Because of the black disks at the borders of certain tiles, we have $c = 1$ if and only if the diagonal signal hits the rightmost tile of the south border of $C_0$, if and only if $k = m-1$.
  
  Suppose now that $k = 0$, so that the east border of $C_0$ holds the value $2$ and $C_0$ does not contain a gray region.
  The southwest corner of $C_0$ must then be the tile at the top right of Figure~\ref{fig:probe-alphabet} with $b_0 \in A \times \{0\}$.
  By rule~\ref{it:prb2}, the west border of $C_0$ holds the symbol $(a, 0)$, where $a \in A$ is the symbol of the base alphabet layer on the same column.
  
  We have now shown that for each row of a simulated tiling of $T_n$ whose probe interface layer contains a word of the form $(a,c) 1^k 2$ with $k \geq 0$, we must have $k < m$, $a$ must equal the symbol on the $A$-layer of the $k$th column of the grid cell simulating the symbol $(a,c)$, and $c = 1$ if and only if $k = m-1$.

  As for the cardinality of the new SFT, note first that if the grid layer contains a finite-width grid, then the probe implementation layer is completely determined by the other layers.
  Conversely, if there is no finite-width grid, then the probe implementation layer too contains no parallel thick lines, and a simple case analysis shows that the number of valid contents for this layer is countable when the other layers are fixed.
\end{example}

% \begin{example}[Sparse finite grids]
%   We describe a tile set that can be used to construct $n$-by-$n$ grids of $n$-by-$n$ squares for any given $n \geq 1$.
%   It consists of three layers, the first of which resembles the grid tile set $\agrid$, but contains a background tile $\#$ and border tiles that can be used to embed a rectangular array of grid cells into a sea of $\#$-symbols.
%   The second layer contains horizontal signals similar to the dashed lines of the probe tile set $\aprb$, which must be present in every grid cell and move to the south by one tile every time a vertical grid line is crossed from west to east.
%   We require these signals to begin at the northwest corner of every grid cell on the west border of a finite grid, and end at the southeast corner of every grid cell on the east border of a finite grid.
%   In this way, each grid will be finite in the east-west direction, and in fact will be exactly $n$ grid cells wide if the grid cells are $n$-by-$n$.
%   The third layer is analogous to the second by rotated by 90 degrees, so it constains the north-south direction.

%   In our application, we will further color each grid cell with a single tile from some set of Wang tiles $T$, and require that the tiles of neighboring grid cells match.
% \end{example}

A weaker version of the following result is stated explicitly, and proved in less detail, in~\cite{To20} as part of the proof of Theorem~1.
The new properties are the tile $t_n$ and the bound $k_n = O(\log \log n)$.
The tileset needs to be modified to obtain them.

\begin{lemma}
  \label{lem:det-tilesets}
  Let $g : \N \to \N$ be a computable function.
  Then there exists a sequence $(C_n, T_n, a_n, b_n, t_n)_{n \in \N}$, computable in polynomial time, with the following properties for all $n \in \N$.
  \begin{enumerate}
  \item
    $T_n$ is an NW-deterministic Wang tileset with edge colors $C_n = \{1, \ldots, k_n\}$, where $k_n = O(\log \log n)$.
  \item
    \label{it:tiles-rect}
    $a_n, b_n \in C_n$ are edge colors and $t_n \in T_n$ a tile such that for some $m \geq g(n)$, $T_n$ tiles an $m \times m$ square whose entire north border is colored with $a_n$, west border with $b_n$, and the tile at the southeast corner is $t_n$.
    Furthermore, $T_n$ does not tile any such rectangle for $m < g(n)$.
  \item
    $T_n$ does not tile the infinite plane.
  \end{enumerate}
\end{lemma}

\begin{proof}
  For $n \in \N$, let $M_n$ be a Turing machine that counts to $g(n)$, enters a special state $q$ that is not used otherwise, counts to $g(n)$ again and then halts.
  We construct the machine in such a way that it has $\lceil \log n \rceil + K$ states for some constant $K$, by having it first write $n$ in binary on its tape and then perform a uniform algorithm with input $n$.
  Let $t(n) \geq 2 g(n)$ be the number of steps $M_n$ takes before halting.
  
  In~\cite{Ka92}, Kari constructs a NW-deterministic version of the aperiodic Robinson tile set~\cite{Ro71} and embeds an arbitrary Turing machine in it.
  Its tilings consist of nested square patterns, called $(2^k-1)$-squares, each of which hosts a simulation of the embedded machine for approximately $k/2$ steps.
  If the machine halts in these $k/2$ steps, the pattern cannot be tiled correctly.
  
  Denote by $R_n$ this tile set with $M_n$ being the embedded machine.
  It can tile a southeast-facing $(2^k-1)$-square with $2^{k-1} \leq t(n)/2 < 2^k$.
  The north border of such a square always has the same color $a_n$, and its west border always has the same color $b_n$.
  The tile that simulates the special state $q$ is chosen as $t_n$; we can build $M_n$ and the simulation so that it always lies on the diagonal of a $(2^n-1)$-square.
  Then item~\ref{it:tiles-rect} is satisfied with the $m \times m$-pattern that spans the northwest corner of the $(2^k-1)$-square and the simulated state $q$ (which takes at least $g(n)$ steps to reach, whence $m \geq g(n)$).
  
  Since the machine $M_n$ eventually halts, $R_n$ does not tile the plane.
  Finally, the number of colors in $R_n$ grows linearly in the number of states and transitions of $M_n$, and we can make the latter grow as $O(\log \log n)$ by replacing $g$ with $g \circ \exp$ and using $M_{\lceil \log n \rceil}$ in place of $M_n$.
  Note that the sequence is allowed to have repeated values.
\end{proof}

The following lemma is a key ingredient of our construction.
It allows us to take a small grid and use it to enforce the existence of a large but still finite grid on another layer.
The idea is implicitly present in the proof of~\cite[Theorem~3]{To20}, but used in an ad hoc manner and not formalized.

\begin{lemma}[Blowing up grids]
  \label{lem:blowup}
  Let $g : \N \to \N$ be a total computable function.
  There exists a finite alphabet $A_g$ and a countable SFT $X_g \subset \grid \times \grid \times A_g^{\Z^2}$ such that for each $n \geq 1$, there exists a nonempty finite set $K \subset \N \cap [g(n), \infty)$ such that
  \[
    \set{ y \in \grid }{ \exists z : (\xgrid(n), y, z) \in X_g } = \set{ \sigma^{\vec v}(\xgrid(k)) }{ k \in K, \vec v \in n \Z^2 },
  \]
  and if $(x, y, z) \in X_g$ is such that $y$ has a grid of finite width, then so does $x$.
\end{lemma}

\begin{proof}
  Let $(C_n, T_n, a_n, b_n, t_n)_{n \in \N}$ be the sequence of tile sets and tiles given by Lemma~\ref{lem:det-tilesets} for the function $n \mapsto g(2^n)$.
  Construct a new sequence of tile sets $(S_n)_{n \in \N}$ as follows.
  Take the grid shift $\grid$, and decorate each tile without a horizontal or vertical line with a tile of $T_n$ in all possible ways.
  Two adjacent decorated tiles must respect the adjacency rules of $T_n$.
  If the north neighbor of a decorated tile $t \in S_n$ is a horizontal line, then the north color of the $T_n$-decoration of $t$ must be $a_n$, and if the west neighbor of $t$ is a vertical line, then the west color of the decoration must be $b_n$.
  If the east neighbor of $t$ is a vertical line and its south neighbor is a horizontal line, then its decoration must be $t_n$.
  The colors of all other edges of the $T_n$-decorations are not constrained.
  This concludes the definition of $S_n$.
  
  The tile set $S_n$ does not admit tilings without a finite grid, since $T_n$ does not tile the infinite plane.
  By compactness, the set of grid sizes it admits is finite.
  It does admit at least one tiling with a grid of some width $m \geq g(2^n)$, but not of any width $m < g(2^n)$.
  Also, the decorations inside each grid cell form the same pattern in every tiling, since $T_n$ is NW-deterministic and the colors of the north and west borders of the grid cells are fixed.
  Thus the set of valid tilings over $S_n$ is countable.

  Similarly to Example~\ref{ex:counter-machine-grids}, we define $X_g \subset \grid \times \grid \times A_g^{\Z^2}$, where $A_g$ is an auxiliary alphabet, as an SFT where each grid cell of the first component contains a simulated computation of a counter machine $M$ on the $A_g$-component.
  We call the three components of $X_g$ the \emph{small grid layer}, the \emph{large grid layer} and the \emph{simulation layer}.
  
  The machine $M$ has six named counters $b, m, t_e, t_n, t_w, t_s$.
  It nondeterministically chooses some value for each counter (with $b \in \{0,1\}$), and checks  whether the quadruple $t = (t_e, t_n, t_w, t_s) \in \N^4$ is a tile of $S_m$.
  As in Example~\ref{ex:counter-machine-grids}, adjacent small grid cells synchronize these simulated counters so that each machine has the same value of $m$ and the colors of the simulated tiles match, so that a configuration containing a finite-width small grid will simulate a configuration of some $S_m$.
  
  After this, the machine $M$ checks that $b = 1$ when $t$ corresponds to a tile of $S_m$ with a horizontal grid line, and $b = 0$ otherwise.
  Let $h(m) \geq m$ be an upper bound for the number of computation steps required thus far for a fixed $m$ and any values of the other named counters.
  By Lemma~\ref{lem:det-tilesets} and the exponential slowdown in simulating a Turing machine by a counter machine, we can choose $h(m) = 2^{m^q}$ for some $q \in \N$.
  We may also assume that the machine always takes the same number of steps for the check when $m$ is fixed.
  For example, we may have it perform the check for all possible values of the counters $t_e$, $t_n$, $t_w$ and $t_s$ that are below the maximum value $\log \log m$ and ignore all results except the one with the actual counter values; the total number of steps is still $2^{n^{O(1)}}$, so this does not violate the previous assumption.
  
  After $M$ has finished the check, it runs for at least $h(m+1)$ more steps, e.g.\ by simply computing $h(m+1)$ and then counting to $h(m+1)$.
  After this it finally halts.
  Note that the time required to compute $h(m+1)$ is irrelevant, since we only need $M$ to run for at least $h(m+1)$ steps but not indefinitely.
  % ; for large enough $m$ this takes at most $m^{q'}$ computation steps for any fixed $q' > q$.
  We allow the simulation of $M$ to be cut by a horizontal small grid line at any point after it has finished its check (that is, after $h(m)$ computation steps) but before it has halted.
  
  For each cross tile $t$ of the small grid, we require that $t$ is matched with a horizontal line on the large grid if and only if the northeast neighbor of $t$ contains a simulated $b$-counter with value $1$ on its $A_g$-component.
  Also, each cross tile of the large grid must be matched with a cross tile on the small grid.
  For finitely many widths of the small grid (including those for which the small grid cells are too small to host a computation), we can enforce the width of the large grid by explicit rules.
  This concludes the definition of $X_g$.
  
  We show that $X_g$ has the desired properties.
  Take any $n \geq 1$.
  Let $y$ and $z$ be such that $z' = (\xgrid(n), y, z) \in X_g$.
  If $n$ is large enough, every small grid cell (which has size $n \times n$) hosts a simulation of $M$ that computes a tile from the same set $S_m$ and enforces the adjacency rules to its neighbors.
  We must have $m \geq \log n$, since $M$ always runs for at least $m$ steps, which corresponds to at least $2^m$ rows in the tiling.
  
  By the second paragraph of this proof, the tiling of $S_m$ simulated on the small grid cells contains a grid of some finite width $w \geq g(2^m)$.
  Hence the small grid layer contains regularly spaced small grid cells that simulate cross tiles.
  These must be matched with horizontal lines on the large grid layer $y$.
  Thus $y$ has a finite-width large grid, that is, it is a shifted version of some $\xgrid(k)$ for $k \in \N$.
  Moreover, the simulated grid of $S_m$ has width $w \geq g(2^m) \geq g(2^{\log n}) = g(n)$, so the large grid of $y$ has width at least $n g(n) > g(n)$ for large enough $n$ (and small values of $n$ are handled separately).
  As the small and large grids must be aligned, we have $y = \sigma^{\vec v}(\xgrid(k))$ for some $k \geq g(n)$ and $\vec v \in n \Z^2$.
  Since the small grid $\xgrid(n)$ is $n$-periodic both horizontally and vertically, $\vec v$ can be chosen arbitrarily.
  Thus there exists some set $K \subset \{g(n), g(n)+1, \ldots\}$ such that the claimed equation holds.
  
  By compactness, the set $K$ must be finite.
  We show that for large enough $n$ it is also nonempty.
  For a given $m$, the simulated machine takes at most $h(m)$ steps to perform its computations, after which it runs for at least $h(m+1)$ more steps before halting, where $h(m) = 2^{m^q}$.
  It takes $C^t$ rows to simulate $t$ computation steps of the machine for a constant $C \geq 2$.
  This means that it is enough to find an $m$ such that $C^{h(m)} \leq n \leq C^{h(m+1)}$, so that the machine has enough space to perform the check in an $n \times n$ square, but not enough to halt.
  This is equivalent to
  \[
    m \leq \left( \log \left( \log_C n \right) \right)^{1/q} \leq m+1.
  \]
  Clearly, such an $m$ can always be found.
  Thus there exist $y$ and $z$ with $(\xgrid(n), y, z) \in X_g$, or in other words, $K \neq \emptyset$.

  Finally, suppose $(x, y, z) \in X_g$ is such that $y$ contains a finite-width grid.
  Since every cross of $y$ is matched with a cross in $x$, the latter must also contain a finite-width grid.
\end{proof}

\section{Proof of Theorem~\ref{thm:main}}
\label{sec:main-constr}

We now combine our puzzle pieces in order to form a proof of Theorem~\ref{thm:main}.
One direction follows directly from Theorem~\ref{thm:ccc}.

For the converse, suppose we are given a nondecreasing computable function $f : \N \to \N$ and a shift space $X \subseteq G(f)$ that satisfies the countable cover conditions.
Our goal is to construct a countable SFT cover for $X^\dag$.
In the first phase of the construction, we embed $X^\dag$ into a countably covered sofic shift $Y$ that is contained in $G(g)^\dag$ for some function $g : \N \to \N$ that dominates $f$.
After this, we will further restrict the SFT cover of $Y$, transforming it into a cover of $X$.

Since the construction is rather complex and contains many interconnected layers, we list here all the relevant layers, their roles, and the place where they are defined.
The reader may use this list as a quick reference.
\begin{itemize}
\item
  Base alphabet layer: $(A^\Z)^\dag$, target of the projection.
  Defined in Section~\ref{sec:embedding}.
\item
  Small grid layer: its width $n$ is at most the minimum distance between nonzero columns of the base alphabet layer.
  Defined in Section~\ref{sec:embedding}.
\item
  Width restriction layer: implements the restriction on the small grid layer.
  Defined in Section~\ref{sec:embedding}.
\item
  Large grid layer: has width $m \geq f(n)+2$ (using Lemma~\ref{lem:blowup}).
  Define $g(n)$ as the maximum width of this layer.
  Defined in Section~\ref{sec:embedding}.
\item
  Counting layer: restricts the number of nonzero columns of the base alphabet layer to be at most $m$ (using Lemma~\ref{lem:count-columns}).
  Defined in Section~\ref{sec:embedding}.
\item
  Simulation grid: has width at least $t(n) \cdot m$ (using Lemma~\ref{lem:blowup}), simulates a tiling over tile set $T_n$.
  Defined in Section~\ref{sec:simulation-layer}.
\item
  Probe interface layer of $T_n$ and probe implementation layer: together allow the simulated tilings over $T_n$ to probe the base alphabet layer.
  Defined in Example~\ref{ex:probe} and Section~\ref{sec:sets-Tn}.
\item
  Layer $L_{n,a,b}$ of $T_n$: responsible for checking the correctness of the base alphabet layer between nonzero columns $a$ and $b$.
  There is one for each $0 \leq a < b \leq f(n)-1$.
  Defined in Section~\ref{sec:sets-Tn}.
\item
  Computation sublayer of $L_{n,a,b}$: simulates a counter machine $C_n$ that controls the probing and correctness checks in $L_{n,a,b}$.
  Defined in Section~\ref{sec:sublayers}.
\item
  Anchor sublayer of $L_{n,a,b}$: has the form $\cdots 0 0 1^k 2 2 \cdots$, where the $1$-symbols cover the area between nonzero columns $a$ and $b$ of the base alphabet layer.
  Defined in Section~\ref{sec:sublayers}.
\item
  Small restart grid sublayer of $L_{n,a,b}$: has width $1 \leq w \leq k$, where $k$ is given by the anchor sublayer.
  Defined in Section~\ref{sec:restarts}.
\item
  Large restart grid sublayer of $L_{n,a,b}$: has width at most $h'(w)$ (using Lemma~\ref{lem:blowup}), where the simulated computation of $C_n$ fits on $h'(k)$ rows.
  Periodically restarts the computation of $C_n$.
  Defined in Section~\ref{sec:restarts}.
\end{itemize}

\subsection{Embedding $X^\dag$ in a countably covered $G(g)^\dag$}
\label{sec:embedding}

We begin to construct a countable SFT $Z$ and a projection $\pi : Z \to A^{\Z^2}$ such that $Y = \pi(Z)$ satisfies $G(f)^\dag \subseteq Y \subseteq G(g)$ for some function $g$.
This SFT consists of several layers, the first of which is the vertically constant $A$-layer $(A^\Z)^\dag$, called the \emph{base alphabet layer}.
We define $\pi$ as the projection to this layer.

The second layer is a copy of $\grid$, which we call the \emph{small grid layer}.
The third layer is the \emph{width restriction layer} given by the tiles visible in Figure~\ref{fig:alph-restr} and their obvious matching rules.
Its thick horizontal lines must coincide with those of the small grid layer, and its vertical lines must coincide with the nonzero columns of the base alphabet layer.
It restricts the width of the small grid layer to be at most the minimum distance between two nonzero columns of the base alphabet layer.
In particular, if there are at least two nonzero columns, then the small grid layer will necessarily contain a grid of finite width (and the width can be any integer between $1$ and the minimum distance).
Note that the width restriction layer is completely determined by the previous layers in the case that there are at least two nonzero columns in the base alphabet layer, and has countably many configurations otherwise.

\begin{figure}[htp]
  \centering

  \begin{tikzpicture}[scale=0.7]

    \pgfmathsetmacro{\ymax}{9}
    \pgfmathsetmacro{\xmax}{16}

    \begin{scope}
      \clip (0,0) rectangle (\xmax, \ymax);
      \foreach \y in {-4,0,4,8}{
        \foreach \x in {1, 5, 11}{
          \draw [fill=black!20] (\x+0.5,\y+0.5) -- ++(4,4) -- ++(0,-4);
          % \fill (\x+0.5,\y+0.5) circle (0.15cm);
        }
        \draw [very thick] (0, \y+0.5) -- (\xmax, \y+0.5);
      }
    \end{scope}

    \draw [dotted] (0,0) grid (\xmax,\ymax);
    
  \end{tikzpicture}
  
  \caption{A sample configuration of the width restriction layer.}
  \label{fig:alph-restr}
\end{figure}

Next, we use Lemma~\ref{lem:blowup} to add another copy of $\grid$, called the \emph{large grid layer}, and couple its width to that of the small grid layer using the function $f + 2$.
If the small grid layer has a grid of some finite width $n \in \N$, now the large grid layer must have a grid of finite width $m \geq f(n) + 2$.
We choose $g(n)$ as the maximum value of $m$, the existence of which is guaranteed by Lemma~\ref{lem:blowup}; in other words, $g(n) = \max K$, where $K \subset [f(n)+2, \infty)$ is the finite set given by the Lemma.

Finally, we use Lemma~\ref{lem:count-columns} to add a counting layer that couples the large grid layer with the base alphabet layer.
It has the effect of constraining the number of nonzero columns on the base alphabet layer to be at most $m-2$, but places no other restrictions on it.

\subsection{Correctness of the embedding}

By the construction and the lemmas used, $Z$ is a countable SFT.
It remains to be shown that $G(f)^\dag \subseteq Y \subseteq G(g)^\dag$.
For the first inclusion, let $y \in G(f)^\dag$.
If $y$ contains at most one nonzero column, then the other layers of $Z$ can be filled by degenerate configurations.
If it contains two or more nonzero columns, let $n > 0$ be the minimum distance between any two of them.
We can use $\xgrid(n)$ as the small grid layer, and then the width restriction layer can be filled legally.
By Lemma~\ref{lem:blowup}, we can use $\xgrid(m)$ for some $m \geq f(n)+2$ as the large grid layer, and since $y \in G(f)^\dag$, this does not violate the coupling given by Lemma~\ref{lem:count-columns}.
Hence $y \in Y$.

For the other inclusion, let $y \in Y$ be arbitrary, and let $z \in Z$ be such that $y = \pi(z)$.
If the small grid layer of $z$ does not contain a finite-width grid, then $y$ contains at most one nonzero column, so $y \in G(g)^\dag$.
Suppose then that the small grid layer contains a grid of width $n$.
Then Lemma~\ref{lem:blowup} guarantees that the large grid layer contains a grid of width $m$ for some $m \in K$, which here implies $m \leq g(n)$ due to our choice of $g$.
Lemma~\ref{lem:count-columns} states that there are at most $m-2$ nonzero columns in $y$.
The width restriction layer ensures that the distance between any two such columns cannot be less than $n$.
Hence $y \in G(g)^\dag$.

\subsection{The simulation layer}
\label{sec:simulation-layer}

We now begin the second phase of the construction.
This involves adding new layers to the SFT cover $Z$ to produce a new countable SFT $Z'$ with $\pi(Z') = X^\dag$.
We first add a new copy of $\grid$ called the \emph{simulation grid}, on which we simulate a sequence of tile sets $(T_n)_{n \in \N}$ as in Example~\ref{ex:probe}; they will have access to the base alphabet layer via the kind of interface defined in that example and are responsible for checking its correctness.
We will constrain $n$ to be equal to the width of the small grid layer.

The size of the simulation grid is at least $t(n) \cdot m$ for a computable function $t : \N \to \N$ growing fast enough that the counter machines always have enough time to correctly simulate tiles of $T_n$ (recall that $m$ is the size of the large grid); we enforce this using Lemma~\ref{lem:blowup}.
We require that the small grid, large grid and simulation grid all be aligned, so that the counter machines running on the simulation grid cells can access $n$ and the contents of the counting layer.

\subsection{The simulated tile sets $T_n$}
\label{sec:sets-Tn}

We define the simulated tile sets $T_n$ for $n \in \N$ as SFTs, which can then be recoded into sets of Wang tiles.
The set $T_n$ contains $k_n = \binom{f(n)}{2}$ Wang tile layers, one layer $L_{n,a,b}$ for each two-element subset $\{a, b\} \subset \{1, \ldots, f(n)\}$ with $a < b$.
We identify these numbers with the nonzero columns of the base alphabet layer as they are counted from left to right by the counting layer; the number of a nonzero column is the height of the shaded region on the counting layer at that position (see Figure~\ref{fig:count-conf}).
In particular, if there are less than $f(n)$ nonzero columns, the count may not begin from 1.
The layer $L_{n,a,b}$ is responsible for probing the contents of the base alphabet layer between the $a$th and $b$th nonzero columns and verifying that this part does not contain forbidden patterns of $X$.

In addition to these layers, the set $T_n$ contains one probe interface layer over the alphabet $B$ of Example~\ref{ex:probe}, which is used to probe the base alphabet layer.
Each layer $L_{n,a,b}$ will have access to this common layer, and through it the contents of the base alphabet layer and the small grid.
For this purpose, we add a probe implementation layer to $Z'$ and couple it to the base alphabet and simulation layers as in Example~\ref{ex:probe}.

The probing interface of Example~\ref{ex:probe} allows only one symbol per row to be probed by a simulated tiling.
Hence, we cannot allow the $k_n$ layers of $T_n$ to access the base alphabet layer independently of each other.
For this reason, exactly one of the layers $L_{n,a,b}$ is \emph{active} on a given horizontal row, and the others are \emph{inactive}.
For a pair of adjacent horizontal rows over $T_n$, the layer that is active on the bottom row must follow the adjacency rules of its layer between these two rows, while the remaining $k_n-1$ layers must be identical across the two rows.
The role of the active layer is passed on cyclically, so that on a height-$k_n$ horizontal strip of tiles, each layer is active on exactly one row.
Hence, each layer is ``stretched'' vertically by a factor of $k_n$.

\subsection{The layers $L_{n,a,b}$ of the simulated tilings}
\label{sec:sublayers}

Each of the $k_n$ layers $L_{n,a,b}$ of $T_n$ consists of several sublayers.
The first sublayer, called the \emph{computation sublayer}, contains a simulation of a deterministic bidirectional counter machine $C_n$, which is periodically restarted.
Its purpose is to repeatedly probe the base alphabet layer of $Z$ between the $a$th and $b$th nonzero column and determine whether it is in $X^\dag$.

The second sublayer, called the \emph{anchor sublayer}, is a vertically constant layer over the alphabet $\{0,1,2\}$, in which the horizontal patterns $10$, $20$ and $21$ are forbidden.
The machine $C_n$ has access to the anchor sublayer.
We will force the $1$-symbols on the anchor sublayer of $L_{n, a, b}$ to span exactly the region delimited by the $a$th and $b$th nonzero columns of the base alphabet layer, so that the counter machine $C_n$ of the computation sublayer can be initialized near the $a$th column and check the configuration between it and the $b$th column.

The positions of the $1$-symbols of the anchor layer are determined by the counter machines on the simulation grid as follows.
The simulation grid is aligned with the small and large grids, and the counter machine in each cell of the simulation grid has access to the contents of them and the counting layer.
In particular, it can read the width $n$ of the small grid, as well as the position of the westmost dashed vertical line on the synchronization layer in its grid cell.
Denote by $e \geq 1$ the distance of the dashed line from the west border of the grid cell; recall that this distance is incremented between two adjacent small grid cells if and only if the western one contains a nonzero column.
The machine checks that $e \leq f(n)$, rejecting if not.
Then, on the anchor sublayer of each layer $L_{n,a,b}$ it places a $0$-symbol if $e < a$, a $1$-symbol if $a \leq e \leq b$, and a $2$-symbol if $b < e$.
In this way we guarantee that every configuration contains at most $f(n)$ nonzero columns, and that each pair of these columns has a corresponding simulated layer on which the space between them is highlighted with $1$-symbols.

\subsection{The counter machine $C_n$ of $L_{n,a,b}$}
\label{sec:machine-Cn}

Recall that the layer $L_{n,a,b}$, and hence both its anchor sublayer and the machine $C_n$, are part of the tiling over $T_n$ simulated on the width-$h(n)$ cells of the simulation grid.
We now explain how the machine $C_n$ of $L_{n,a,b}$ accesses the contents of the base alphabet layer.
It has two special counters $c_1$ and $c_2$.
The (possibly empty) segment of $1$-symbols of the common probe interface layer of $T_n$ must always coincide with the portion of the computation cone of $C_n$ that is between the ends of these counters on the active layer.
When the zig-zag head of the active layer steps on the $A \times \{0,1\}$-symbol of the probe interface layer, it can store it as part of its state.

The machine $C_n$ behaves as follows.
First, it determines the number $m$ of $1$-symbols on the anchor sublayer of $L_{n,a,b}$.
Then, using the probe interface layer via its special counters $c_1$ and $c_2$, it probes for the width $h(n)$ of the simulation grid\footnote{Note that $h(n)$ is analogous to the width $m$ in Example~\ref{ex:probe}, where we explain how the simulated tile set can obtain it using the probe interface layer.} and then the full contents of the base alphabet layer on each simulation grid cell that simulates a $1$-symbol on the anchor layer.
In this way it receives a word $w \in A^{m \cdot h(n)}$.
It checks whether $|w|_{\neq 0} \leq f(d)$, where $d$ is the minimum distance between two nonzero symbols in $w$, rejecting (and producing a tiling error) if this is not the case.

Next, $C_n$ checks whether ${}^\infty 0 w 0^\infty \in X$, which is possible since $X$ satisfies the countable cover conditions.
If this holds, then $C_n$ halts and accepts.
If ${}^\infty 0 w 0^\infty \notin X$, then by compactness there exists a number $p \geq 0$ such that $0^p w 0^p \notin \lang(X)$, which we can also compute since $X$ is effectively closed.
After this, the machine probes the $p$ symbols to the left and right of $w$ on the base alphabet layer.
If the result is $0^p w 0^p$, then $C_n$ rejects and produces a tiling error; otherwise, it accepts.

\subsection{Restarting the machine $C_n$: the restart grids}
\label{sec:restarts}

We add two more sublayers to each of the $k_n$ layers of $T_n$, called the \emph{small and large restart grids}, the purpose of which is to periodically restart the computation of the machine $C_n$ on the computation layer.
Namely, if we allowed the machine to run indefinitely, then by shifting vertically we would obtain limit configurations where no computation is ever initialized.

As the names suggest, both sublayers are copied of the grid shift $\grid$.
The small restart grid is coupled with the run of $1$-symbols on the anchor sublayer similarly to the width restriction layer: if $0 1^m 2$ occurs on the anchor sublayer, the small restart grid must have a grid of width between $1$ and $m$.
The large restart grid is coupled to the small restart grid using Lemma~\ref{lem:blowup} with a computable function $h' : \N \to \N$ such that the simulated counter machine $C_n$ will always take at most $h'(m)$ rows of tiles to finish its computation if the number of $1$-symbols on the anchor layer is $m$; again, such a function clearly exists.
A vertical line of the large restart grid erases whatever computation was taking place on the computation sublayer, and initializes a new computation on each occurrence of the word $01$ on the anchor layer (of which there can be at most one).

Finally, there may be symbols $a \in A_0$ such that ${}^\infty 0 a 0^\infty \notin X$.
These cannot be completely handled like the words $w$ such that ${}^\infty 0 w 0^\infty \notin X$ at the end of Section~\ref{sec:machine-Cn}, since they may occur as rows of the base alphabet layer in configurations without a finite-width small grid (and hence without simulations of any $T_n$ and $C_n$).
For each such symbol, there necessarily exists $p \geq 0$ such that $0^p a 0^p \notin \lang(X)$, and we explicitly forbid this horizontal pattern from the base alphabet layer of $Z'$.
This concludes the construction of $Z'$.

\subsection{The SFT $Z'$ is countable}

We claim that $Z'$ is countable.
As $Z$ was already proved countable, it suffices to fix all the layers of $Z$ (the base alphabet layer, the small and large grids, and their width restriction and counting layers) and show that the remaining layers have countably many choices.

If the small grid layer contains no finite-width grid, then neither does the simulation grid layer, by the last claim on Lemma~\ref{lem:blowup}.
Hence, it contains at most one computation cone of a counter machine, which has a countable number of possible computation paths, plus a countable number of choices for other degenerate signals.
Suppose then that the small grid layer contains a grid of width $n$, so that the large grid layer contains a grid of some width $m \geq f(n)+2$.
Then the simulation grid layer contains a grid of width at least $t(n) \cdot m$, on which a tiling of $T_n$ is simulated (since the counter machines in the grid cells can access the value of $n$).

On each of the $k_n$ layers $L_{n,a,b}$ of $T_n$, the anchor layer is completely determined by the previously described layers.
If the anchor layer does not contain a finite run of $1$-symbols, then the small restart grid does not contain a finite-width grid, and neither does the large restart grid by Lemma~\ref{lem:blowup}.
Hence the simulation sublayer contains at most one position where a computation of $C_n$ is initialized, and since $C_n$ is deterministic, this (together with degenerate computation cones) constitutes a countable number of choices.
On the other hand, if the anchor layer contains a pattern $0 1^k 2$, then the small restart grid contains a grid of width at most $k$, and the large restart grid contains a grid of width at most $h'(m)$.
The latter grid periodically restarts the computation of $C_n$ on the same column.
Hence the number of choices is again countable.

\subsection{The SFT $Z'$ projects onto $X^\dag$}

We now prove $\pi(Z') = X^\dag$.
Take any $x \in X^\dag$ and use it as the base alphabet layer.
Since $x \in G(f)^\dag$, we can safely fill all the layers of $Z$, in particular putting a grid of finite width $n$ on the small grid layer, equal to the size of the smallest gap on the base alphabet layer, and one of some width $m$ on the large grid layer.
We choose to fill the counting layer so that the dashed horizontal lines are directly above the horizontal grid lines on the left end of the configuration.
On the simulation grid layer we can safely put a grid of finite width at least $t(n) \cdot m$ by Lemma~\ref{lem:blowup}.
By our choice of the function $t$, the counter machines inside these grid cells have enough time to correctly simulate a tiling of $T_n$.

Consider a layer $L_{n,a,b}$ of $T_n$.
The anchor sublayer of $L_{n,a,b}$ is determined by the previous layers.
If it does not contain a finite run of $1$-symbols, then the remaining sublayers can safely be filled with degenerate, uniform configurations.
If it does contain a run $0 1^k 2$, we can choose a grid of width $k$ as the small restart grid, and a grid of some width at least $h'(k)$ as the large restart grid.

This leaves the computation sublayer (and the common probe interface layer, which is determined by the previous layers and only constrains the computation sublayers).
On this sublayer we are forced to initialize computations of the machine $C_n$ on the intersections of the horizontal grid lines of the large restart grid and the leftmost $1$-column on the anchor sublayer.
Since the base alphabet layer comes from $x \in X$, the machine $C_n$ cannot reject; the only way to do so would be to find too many nonzero columns or a forbidden word of the form $0^p w 0^p$, neither of which can occur on the base alphabet layer.
Hence we have produced a valid tiling of $Z'$.

Conversely, take any $z \in Z'$; we claim that its base alphabet layer is in $X^\dag$.
If the small grid layer is degenerate, then so is the large grid layer, and the base alphabet layer may contain at most one nonzero column.
Since we explicitly forbade any ${}^\infty 0 a 0^\infty \notin X$ from occurring on the base alphabet layer, this case is handled.

Suppose hence that the small grid layer contains a grid of some width $n$.
Then, as in the argument for the countability of $Z'$, we must have finite-size grids on the large grid and simulation grid layers, the latter of which must simulate a tiling of $T_n$, which in turn consists of the $k_n$ layers $L_{n,a,b}$ plus the shared probe interface layer.
The simulation layer also restricts the vertical dashed lines on the counting layer of $Z$ to be at height at most $f(n)$, and assigns to each nonzero column of the base alphabet layer a unique successive number in $\{0, \ldots, f(n)\}$.
Let these numbers span a sub-interval $\{a, a+1, \ldots, b\}$.

We restrict our attention to the layer $L_{n,a,b}$ of the simulated $T_n$-tiling.
The run of $1$-symbols of its anchor sublayer is finite and spans the entire nonzero part of the base alphabet layer.
Thus, the restart grids contain grids of finite size, and the computation sublayer contains an infinite number of restarted computations of $C_n$ that have enough time to complete by accepting.
The machine $C_n$ reads the entire nonzero part of the base alphabet layer as the word $w$.
If we had ${}^\infty 0 w 0^\infty \notin X$, then $C_n$ would have to find a nonzero column within a finite distance to the left or right of $w$, which is impossible, since $w$ spans all the nonzero columns.
Hence we have ${}^\infty 0 w 0^\infty \in X$, as required.
This finishes the proof of Theorem~\ref{thm:main}.

\subsection{A modification}

With minor modifications to the proof of Theorem~\ref{thm:main} we can extend it to a class of one-dimensional shift spaces that includes all subshifts of both one-dimensional sofic shifts and gap shifts.
Namely, fix an alphabet $A$, a computable function $f : \N \to \N$ and a bound $b > 0$.
In each configuration of the gap width shift $G(f, \{0,1\})$, replace each symbol within distance $b$ from a $1$-symbol with an arbitrary symbol of $A$, and then replace each remaining (finite or infinite) run of $0$-symbols with some equal-length periodic pattern over $A$ with period at most $b$.
Each of these replacements is done independently of the others.
Denote the set of all possible resulting configurations by $G(A, f, b)$.

We claim that subshifts of such $G(A, f, b)$ also satisfy the converse of Theorem~\ref{thm:ccc}.
In the proof, we add a new $\{0,1\}$-layer on which we place a configuration of $G(f, \{0,1\})$ from which the $A$-layer could have been produced.
This can be verified with only local rules.
We use this $\{0,1\}$-layer in place of the $A$-layer in every part of the construction, except that we give the $C_n$-machines access to the $A$-layer in order to verify its correctness.
The algorithm of $C_n$ is modified suitably.

\section{Proof of Theorem~\ref{thm:main2}}
\label{sec:main2}

In this section we prove Theorem~\ref{thm:main2}.
Assume thus that $f : \N \to \N$ is nondecreasing and upper semicomputable, and satisfies $f(n) < 2^{\sqrt{n}}$ for all large enough $n$.
We construct a countable SFT $X$ that factors onto $G_2(A,f)$.

\subsection{Small $n$}

We first show that small values of $n$ can be handled separately.
Suppose that $f(n) < 2^{\sqrt{n}}$ holds when $n > M$.
To handle all $1 \leq n \leq M$, we use the alphabets $A_n = \{(n,k,d,a) \mid 0 \leq k \leq f(n), d \in \{-1,1\}, a \in A\}$.
The idea is that in a configuration over $A_n$, the $k$-value keeps track of the number of nonzero symbols seen when traversing the rows one by one, alternating left-to-right and right-to-left.
The $d$-value is constant on each row and alternates on each column, encoding the direction of travel.

Suppose that $x \in X$ and $x_{\vec v} = (n,k,d,a) \in A_n$.
Then all neighbors of $\vec v$ must also have symbols of $A_n$.
We also require $x_{\vec v + (0,1)} = (n,k',-d,a')$ with $k' \geq k$, and $x_{\vec v + (d,0)} = (n, k'', d, a'')$ with $k'' \geq k$, and $k'' > k$ if $a \neq 0$.
Finally, if $a \neq 0$, we require that no other symbol at distance less than $n$ has a nonzero $a$-component.
The factor map is defined as $(n,k,d,a) \mapsto a$ on $A_n$.

We claim that $X \cap A_n^{\Z^2}$ maps onto the subshift $Y_n$ of $G_2(A,f)$ containing all configurations with at most $f(n)$ nonzero symbols, no two of which are at distance less than $n$ (note that $G_2(A,f) = \bigcup_{n \geq 0} Y_n$).
Let first $x \in Y_n$, so that the total number of nonzero symbols is at most $f(n)$.
We construct a configuration $y \in X \cap A_n^{\Z^2}$ that maps to $x$.
Consider the linear order on $\Z^2$ given by $(i,j) \leq_\pm (i',j')$ if either $j < j'$ or $j = j'$ and $(-1)^j i \leq (-1)^j i'$.
For $\vec v = (i,j) \in \Z^2$, let $y_{\vec v} = (n, k, (-1)^j, x_{\vec v})$, where $k$ is the number of positions $\vec w \leq_\pm \vec v$ such that $x_{\vec w} \neq 0$.
Then $y \in X$: if $y_{\vec v} = (n, k, d, a)$, then $\vec v \leq_\pm \vec v + (0,1)$ and $\vec v \leq_\pm \vec v + (d,0)$, so the $k$-values are valid, and the last condition holds since $x \in Y_n$.

Conversely, suppose $y \in X \cap A_n^{\Z^2}$ and let $x \in A^{\Z^2}$ be its image.
By translating, we may assume that $y_{(0,0)} = (n, k, 1, a)$, and then the $d$-value of $y_{(i,j)}$ is $(-1)^j$ for all $(i,j) \in \Z^2$.
By the constraints of $X \cap A_n^{\Z^2}$, the $k$-value of $y_{\vec v}$ is nondecreasing with respect to the linear order ${\leq_\pm}$, and has to properly increase at every $\vec v$ such that $x_{\vec v} \neq 0$.
In particular, the total number of nonzero symbols in $x$ is at most $f(n)$.
Also, no two nonzero symbols in $x$ are at distance less than $n$.
Hence $x \in Y_n$.

The subshift $X \cap A_n^{\Z^2}$ is also countable, since every configuration defines a linear order on $\Z^2$ out of two possibilities, and along this order the value of $k$ is nondecreasing.

\subsection{The layers}

Let us now handle large values of $n$.
Configurations of this part of $X$ consist of five layers: the \emph{base alphabet layer} which is simply $A^{\Z^2}$, the \emph{base grid}, the \emph{binary counter layer}, the \emph{distance layer}, and the \emph{computation layer}.
The factor map projects every configuration to its base alphabet layer.

The base grid is a copy of the grid shift $\grid$.
Its purpose is to provide an ``anchor'' for the rest of the construction, and its width will be approximately equal to the minimum max distance $d$ between two nonzero symbols on the base alphabet layer.
The binary counter layer counts the number $N$ of nonzero symbols on the base alphabet layer, and the distance layer measures the exact minimum distance $d$.
Both of them transmit their respective information to the computation layer, which simulates a computation of a counter machine that ensures $N \leq f(d)$.

\subsection{The base grid}

The base grid is a copy of $\grid$, with some additional colors overlaid on the interior tiles of grid cells.
Each grid cell is completely colored with either $0$ or $1$, depending on whether it contain a nonzero symbol on the base alphabet layer.
This is enforced as follows.
If a grid cell $C$ is colored with $1$, then we place inside it a colored rectangle whose southwest corner coincides with that of $C$ (the two colors ``rectangle interior'' and ``rectangle exterior'' are overlaid on the color $1$).
The northeast corner of the rectangle must contain a nonzero symbol on the base alphabet layer.
Conversely, each nonzero symbol on the base alphabet layer must be at the northeast corner of such a rectangle.
If $C$ is colored with $0$, it cannot contain such a rectangle.

We further require that two $1$-colored grid cells cannot share a border or corner.
This ensures that if the minimum distance between two nonzero symbols on the base alphabet layer is $d$, then the base grid has width at most $d$, and on the other hand we can place a valid base grid of width $\lfloor d/2 \rfloor$.

In the rest of the construction we will define ``signals'' traveling inside cells of the base grid.
Such a signal is always implemented as the border between two distinctly colored regions of a grid cell, using two new colors that are overlaid on any already existing colors and do not interact with them, unless otherwise noted.
The purpose of this is to ensure countability by avoiding limit configurations that contain nothing but infinitely many parallel signals.
% As in the proof of Theorem~\ref{thm:main}, the purpose of the base grid is to anchor the rest of the construction to (a value close to) the distance $d$.

\subsection{Counting nonzero symbols: the binary counter layer}

The binary counter layer is constructed on top of the base grid.
% Its purpose is to count the number of nonzero symbols on the base alphabet layer and transmit this information to the computation layer.
We will superimpose a binary counter on each row of the base grid, which is incremented at each position where the grid cell contains a nonzero symbol.
There is another binary counter on each column of the base grid, on which we will produce the cumulative sum of the counters of the rows crossing it.
We cannot encode the counter at a rate of one cell per bit, as that would result in an uncountable number of limit configurations, so the number of cells per bit has to grow with the length of the counter.
The bound $f(n) < 2^{\sqrt{n}}$ comes from this part of the construction.
With a more efficient encoding we could loosen the bound, but even at one cell per bit (at which point we lose countability) it would remain as $f(n) \leq 2^n$.

On each grid cell $C$ of the base grid we superimpose a secondary square grid of $k \times k$ cells on the binary counting layer, for some $1 \leq k \leq n$.
The southwest corner of $C$ is aligned with a grid cell of the secondary grid, and the secondary grids of adjacent cells of the base grid must have equal width.
The northmost row and eastmost column of the secondary grid might be truncated.

We also enforce that if the secondary grid cells (except those on the north and east borders of $C$) have width $w$, then $k \leq 5w$.
This can be achieved by sending a horizontal signal from the southwest corner of $C$ to the west and incrementing it whenever it has crossed five vertical borders of the secondary grid, similarly to Lemma~\ref{lem:count-columns}.
See Figure~\ref{fig:secondary-grid} for an illustration.
This ensures $w \leq \sqrt{n/5}$, so that $k \leq \sqrt{5n}$.
On the other hand, we can always safely choose $w = \lfloor \sqrt{n/5} \rfloor$, and then $k = \lceil n / w \rceil \geq \sqrt{n/5}$.

\begin{figure}[htp]
  \centering
  \begin{tikzpicture}[scale=0.7]

    \pgfmathsetmacro{\gap}{0.7}
    \pgfmathsetmacro{\gb}{8-\gap}
    \pgfmathsetmacro{\fgap}{5*\gap}
    \pgfmathsetmacro{\rb}{8-5*\gap}
    \foreach \dx in {1,9}{
      \begin{scope}
        \clip (\dx,1) rectangle ++(8,8);
        \foreach \x [count=\n] in {0,\gap,...,8}{
          \foreach \y [count=\n] in {0,\gap,...,8}{
            \fill [black!20] (\dx+\x,1+\y) -- ++(\gap,0) -- ++(0,\gap);
          }
        }
      \end{scope}
      \fill [black!60,opacity=0.5] (\dx,1) -- ++(0,0.15) -- ++(\fgap,0) -- ++(0,0.15) -- ++(\fgap,0) -- ++(0,0.15) -- (\dx+8,1.45) -- ++(0,-0.45);
      \foreach \x [count=\n] in {0,\gap,...,8}{
        \draw (\dx+\x,1) -- (\dx+\x,9);
        \draw (\dx,1+\x) -- (\dx+8,1+\x);
      }
      \foreach \x [count=\n] in {0,\fgap,...,\rb}{
        \draw (\dx+\x,1+0.15*\n) -- ++(5*\gap,0);
      }
      \draw (\dx+2*\fgap, 1.45) -- (\dx+8, 1.45);
    }

    \draw [very thick] (0,1) -- (18,1);
    \draw [very thick] (0,9) -- (18,9);
    \draw [very thick] (1,0) -- (1,10);
    \draw [very thick] (9,0) -- (9,10);
    \draw [very thick] (17,0) -- (17,10);
    
  \end{tikzpicture}
  \caption{The secondary grids of two adjacent base grid cells.}
  \label{fig:secondary-grid}
\end{figure}

Each cell of the secondary grid is colored with an element of $\{0,1\}^2$, that is, two bits.
The first bits encode, in binary, some number $0 \leq a_C < 2^k$ on each vertical column of the secondary grid of $C$, and the second bits encode a number $0 \leq b_C < 2^k$ on each horizontal row.
Let $C'$ be the west neighbor of $C$ and $C''$ the east neighbor of $C$.
If $C$ is colored with $1$ (that is, contains a nonzero $A$-symbol), then $a_C = a_{C'}+1$, and otherwise $a_C = a_{C'}$.
Diagonal signals, implemented as additional colors of the grid cells of the secondary grid, transmit the binary representation of $a_C$ to the southmost row of the secondary grid, where a local rule ensures $b_C = b_{C''} + a_C$.
Both computations can be implemented with simple transducers, using additional colors to encode the carry.
The results of the computations cannot reach $2^k$ or beyond: if the most significant bit produces a carry, a tiling error results.

We also check whether $a_C \neq 0$ and mark it as part of the background color of $C$.
A base grid cell $C$ that satisfies $a_C \neq 0$ is called \emph{active}.
Note that if a cell is active, then so are all cells to the east of it on the same row.

Now each row $R$ of the base grid, when traversed west-to-east, carries a binary counter with some initial value $a_R$ and final value $a'_R = a_R + N_R$, where $N_R$ is the number of nonzero symbols within $R$.
Likewire, each column $C$, when traversed south-to-north, carries a counter with some initial value $a_C$ and final value $a'_C = a_C + \sum_R b_R$, where $b_R$ is the counter value of row $R$ at the position where it crosses $C$.
Since all these values are less than $2^k$, the total number $N$ of nonzero symbols in the configuration is also less than $2^k$.
Since $N \leq 2^{\sqrt{d}}$ holds by assumption and we can always choose $k \geq \sqrt{5n}$ and $n = \lfloor d/2 \rfloor$, this does not unnecessarily restrict $N$.
Furthermore, if we choose $a_R = 0$ for every row $R$, then the final value of every column $C$ far enough to the east is $a'_C = N$, and no counter has a value higher than $N$.

\subsection{Finding the minimum distance: the distance layer}

Via the binary counter layer, the computation layer has access to the total number $N$ of nonzero symbols on the base alphabet layer (or more precisely, a number that is at least $N$).
We will now transmit to it the minimum distance $d$ between two nonzero symbols.
This will be achieved by an arrangement of signals on the distance layer.
Figure~\ref{fig:dist-gadget} contains a diagram of the construction, which the reader should consult to better understand the explanation.

\begin{figure}[htp]
  \centering
  \begin{tikzpicture}[scale=1.7]

    \pgfmathsetmacro{\nzx}{3.6}
    \pgfmathsetmacro{\nzy}{0.8}
    \pgfmathsetmacro{\hgt}{1.7}
    \pgfmathsetmacro{\dq}{4.6}

    \draw [dashed,fill=black!20] (\nzx-\hgt,\nzy) rectangle ++(2*\hgt, \hgt);
    \draw [fill=black!40] (\nzx,\nzy) -- ++(3-\nzx,\nzx-3) -- ++(\hgt,0) -- (\nzx+\hgt,\nzy) -- cycle;

    \foreach \x in {0,...,7}{
      \draw [very thick] (\x,0) -- ++(0,4);
    }
    \foreach \y in {0,...,4}{
      \draw [very thick] (0,\y) -- ++(7,0);
    }

    \fill (\nzx,\nzy) circle (0.04cm);
    \fill (\nzx-0.3,\nzy+\hgt) circle (0.04cm);

    \draw [dashed] (\nzx,\nzy) -- ++(\hgt, \hgt);
    \draw [dashed] (\nzx,\nzy) -- ++(-\hgt, \hgt);

    % \draw [dotted] (0,\nzy) -| (\nzx-\hgt,0);
    % \draw [dotted] (0,\nzy+\hgt) -| (\nzx-\hgt,4);
    % \draw [dotted] (7,\nzy) -| (\nzx+\hgt,0);
    % \draw [dotted] (7,\nzy+\hgt) -| (\nzx+\hgt,4);

    % \draw [dotted] (\nzx,\nzy) -- ++(\nzy,-\nzy);
    % \draw [dotted] (\nzx+\hgt,\nzy) -- ++(\nzy,-\nzy);

    \draw [very thick, dotted] (\dq,0) -- ++(0,4);

    \draw [->] (\nzx-\hgt-0.1,\nzy) |- (\nzx-0.3,\nzy+\hgt+0.1);

    \draw [decorate,decoration={brace,amplitude=0.2cm}] (\nzx+\hgt+0.2,\nzy+\hgt) -- node [midway, right=0.15cm] {$p$} ++(0,-\hgt);
    \draw [decorate,decoration={brace,amplitude=0.2cm}] (\nzx-\hgt,\nzy+\hgt+0.3) -- node [midway, above=0.25cm] {$2p+1$} ++(2*\hgt,0);

    \node [below left] at (\nzx,\nzy) {$\alpha$};
    \node [below right] at (\nzx+\hgt,\nzy) {$\beta$};
    \node [above right] at (\nzx+\hgt-\nzx+3,\nzy+\nzx-3) {$\gamma$};
    \node [above right] at (3,\nzy+\nzx-3) {$\delta$};
    \node at (3.5,0.25) {$C$};
    
  \end{tikzpicture}
  \caption{The structure of a $7 \times 4$ block of base grid cells on the distance layer.}
  \label{fig:dist-gadget}
\end{figure}

Consider a $7 \times 4$ block $B$ of base grid cells such that the central grid cell of the bottom row, say $C$, contains a nonzero symbol (the black dot marked $\alpha$ in Figure~\ref{fig:dist-gadget}).
We will measure the max distance from this nonzero symbol to the closest other nonzero symbol in the block.
Each such block will have its own set of signals that do not interact with those of other blocks; this is achieved by splitting the alphabet of the distance layer into $7 \cdot 4 = 28$ sublayers, one for each $7 \times 4$ block that a given base grid cell belongs to.

In the block $B$ we place a vertical signal called the \emph{distance signal} of $C$ (the dotted line in Figure~\ref{fig:dist-gadget}).
The distance between the west border of $C$ and the distance signal (to its east) encodes a single integer $1 \leq q \leq 2n+1$.
With additional horizontal and diagonal signals, we ensure that every cell of the base grid encodes the same integer.
We will ensure that $q \leq d$, where $d$ is the actual minimum max distance between two nonzero symbols on the base alphabet layer.

We place a rectangular pattern $R$ of size $(2p+1) \times p$ (the gray rectangle in Figure~\ref{fig:dist-gadget}; the shape is enforced by diagonal signals) in the block so that the nonzero symbol of $C$ lies at the center of the bottom row of $R$.
For now, the height $p \geq 1$ can be arbitrary as long as $R$ fits inside the block $B$, but we will constrain it in the following paragraph.
The sides of the rectangle continue to the borders of the block, dividing it into 9 regions with distinct colors.

We require that the interior of $R$ contains no nonzero symbols, and ensure that some coordinate on the west, north or east border of $R$ either contains another nonzero symbol or lies on the border of the block $B$.
This is achieved by emitting a signal from the southwest corner of $R$ that travels clockwise around its border and stops when either condition holds (the arrow in Figure~\ref{fig:dist-gadget}).
If the signal reaches the southeast corner of $R$, a tiling error is produced.
The information about which condition stopped the signal is propagated throughout the block $B$ as an additional background color.

In the case that the signal is stopped by a nonzero symbol, we transmit the height $p$ of the rectangle $R$ -- which is exactly the max distance between the two nonzero symbols -- next to the distance signal of $C$ in order to perform a comparison.
We do this by drawing a parallelogram $\alpha \beta \gamma \delta$ (the dark gray region in Figure~\ref{fig:dist-gadget}), where $\alpha \beta$ codincides with the right half of the south border of $R$, $\delta$ lies on the same vertical line as the west border of $C$, and $\alpha \delta$ and $\beta \gamma$ have slope $-1$.
Again, the sides of the parallelogram continue to the borders of the block, defining 9 regions with distinct colors, all overlaid on top of the already defined colors.
Finally, we require that the distance signal of $C$ crosses the segment $\gamma \delta$, which ensures $q \leq p$.

\subsection{Verifying the constraint: the computation layer}

We now describe the computation layer.
It consists of finite and infinite computation cones (as in Example~\ref{ex:counter-machines}) that host simulated computations of a counter machine $M$, which has access to the contents of the other layers.
Our goal is to ensure that every configuration containing the base grid and at least one nonzero symbol also contains at least one infinite computation cone.

The vertical west border of each computation cone must lie on a vertical line of the base grid.
We place the vertex of a computation cone at the southwest corner of every base grid cell that contains a nonzero symbol on the base alphabet layer.
If the diagonal east border of a computation cone $C_1$ reaches the vertical west border of another cone $C_2$, a signal is emitted from that position directly to the west which truncates the cone $C_1$.
Its computation does not continue.
If a computation cone contains the southwest corner of an active base grid cell, then two signals are emitted from that corner to the west and east which truncate the cone.
In this case, a new cone is created at the intersection of the westward signal and the vertical border of the original cone.

The machine $M$ has access to the base grid, the binary counter layer and the distance layer.
In particular, it can read the value $b = b_C$ from each base grid cells $C$ on the west border of its cone -- all of which have the same value, since a computation cone cannot contain active grid cells -- and the value $q$ stored in the distance signal.
Let $g : \N^2 \to \N$ be a computable upper approximation to $f$, that is, $f(n) = \min_{k \in \N} g(n,k)$ for all $n \in \N$.
The machine $M$ enumerates all $k \in \N$ and checks that $b \leq g(q,k)$, producing a tiling error if not.
Since $q \leq d$ and $f$ is nondecreasing, this implies $N \leq b \leq \min_{k \in \N} g(q,k) = f(q) \leq f(d)$.

\subsection{The SFT $X$ projects into $G_2(A, f)$}

We assume $x \in X$ and claim that the projection of $x$ is in $G_2(A, f)$.
If $x$ is over $A_n$ for some $n \leq M$, this is clear.
Otherwise, if the base alphabet layer of $x$ contains at most one nonzero symbol, this is also clear.
Let thus $d$ be the minimum max distance between two nonzero symbols on the base alphabet layer.
The small grid has some width $n \leq d$, and the number $q$ stored in the distance signals verifies $1 \leq q \leq 2n+1$.

We claim that $q \leq d$.
Consider two nonzero symbols at $(i,j), (i',j') \in \Z^2$ with distance exactly $d$ satisfying $j \leq j'$.
Let $C$ be the small grid cell containing $(i,j)$, and let $B$ be the $4 \times 7$ block of small grid cells containing $C$ at the middle of its bottom row.
Inside $B$ on the distance layer, the nonzero symbol at $(i,j)$ produces a $(2p+1) \times p$ rectangle $R$.
By construction, the interior of $R$ contains no nonzero symbols, but the west, north or east border of $R$ either contains another nonzero symbol or touches the border of $B$.
In the first case, the nonzero symbol at $(i',j')$ lies on the border of $R$, since it is the closest nonzero symbol to $(i,j)$ and $j \leq j'$.
Then the other signals of the distance layer enforce $q \leq d$.
If $R$ borders the $7 \times 4$ block $B$ instead, then $p \geq 3n$, and since $(i',j')$ is not inside $R$, the max distance between $(i,j)$ and $(i',j')$ is $d \geq p \geq 3n > 2n+1 \geq q$.
In both cases the claim holds.

Let $(i,j) \in \Z^2$ be the lexicographically largest coordinate of a nonzero symbol.
We claim that some coordinate $(i',j')$ such that $i \leq i'$ and $j'$ is above every row with active base grid cells is the vertex of an infinite computation cone.
First, the computation layer contains a computation cone $C$ whose vertex is at $(i,j)$.
Suppose $C$ is finite.
If its north border intersects either an active base grid cell, then the northwest corner of $C$ is the vertex of another computation cone.
We replace $C$ by this new cone, which has a lexicographically larger vertex, and repeat the analysis.
Otherwise, $C$ intersects another computation cone to the east.
Since every base grid cell to the east of $(i,j)$ on the same row is active, this new cone has a vertex (either on the same row as $C$ or above it) that is lexicographically larger than that of $C$.
We replace $C$ by the new cone and repeat the analysis.

See Figure~\ref{fig:process} for an example of this process.
The black dots are nonzero coordinates, which produce active base grid cells represented by dashed lines.
The gray triangles are computation cones; the bottom right rectangle is an infinite cone shaped like a quarter-plane.
The dark gray triangles are the ones that the process considers in this particular case.

\begin{figure}[htp]
  \centering
  \begin{tikzpicture}

    \pgfmathsetmacro{\maxx}{6}
    \pgfmathsetmacro{\maxy}{6}

    \fill [black!20] (3.5,-1) rectangle (\maxx,0);
    \draw (3.5,-1) |- (\maxx,0);

    \fill [black!20] (0,5) -- ++(2,2) -- ++(-2,0);
    \draw (0,7) -- (0,5) -- ++(2,2);
    \fill [black!40] (3.5,5) -- ++(2,2) -- ++(-2,0);
    \draw (3.5,7) -- (3.5,5) -- ++(2,2);

    \draw [fill=black!40] (3.5,3) -- ++(2,2) -- ++(-2,0) -- cycle;
    \draw [fill=black!20] (0.5,1.5) -- (0.5,0) -- ++(1.5,1.5) -- ++(-1.5,0);
    \draw [fill=black!40] (2,4.5) -- (2,1) -- ++(1,1) -- ++(-1,0) -- ++(1,1) -- ++(-1,0) -- ++(1.5,1.5) -- cycle;
    \draw [fill=black!20] (-0.5,3) -- ++(2,2) -- ++(-2,0) -- cycle;
    \draw [fill=black!20] (-2,2) -- ++(1.5,1.5) -- ++(-1.5,0) -- cycle;
    \draw [fill=black!20] (-0.5,5) -- ++(0.5,0.5) -- ++(-0.5,0) -- cycle;

    \foreach \x/\y in {0.5/0, 2/1, -2/2, -0.5/3, 0/5}{
      \fill (\x,\y) circle (0.075cm);
      \draw [dashed] (\x,\y) -- (\maxx,\y);
    }
    
    \draw [fill=black!20] (3.5,0) -- ++(1,1) -- ++(-1,0) -- ++(1,1) -- ++(-1,0) -- ++(1,1) -- ++(-1,0) -- cycle;

    \node [below] at (2,1) {$(i,j)$};
    \node [above left] at (3.5,5) {$(i',j')$};
    
  \end{tikzpicture}
  \caption{The process for finding an infinite computation cone.}
  \label{fig:process}
\end{figure}

We claim that the above process eventually terminates.
Namely, the second condition can occur only once.
Since there are no nonzero coordinates to the east of $(i,j)$, the vertical border of the new cone must continue indefinitely to the south (possibly as the vertical borders of yet other cones).
When it has crossed the finitely many rows containing active base grid cells, the remaining half-line $L$ is the border of an infinite computation cone shaped like a southeast quarter-plane.
If the second condition recurred, there would be another such half-line to the east of $L$, which is impossible.
The first condition occurs exactly as many times as there are rows above $j$ with active base grid cells, which is less than $2^k$.
Hence the existence of $(i',j')$ is proved.

The number $b$ received by the machine $M$ of the $(i',j')$-based cone satisfies $b \geq N$, as it lies above every row with active base grid cells and to the east of every nonzero coordinate.
The machine checks that $b \leq g(q,k)$ holds for all $k \in \N$.
Since $f$ is nondecreasing, we have $N \leq b \leq \min_{k \in \N} g(q,k) = f(q) \leq f(d)$, and hence the base alphabet layer is in $G_2(A,f)$.

\subsection{The SFT $X$ projects onto $G_2(A, f)$}

Let $y \in G_2(A,f)$ be arbitrary.
We show that there exists $x \in X$ that projects to $y$.
If $y$ contains at most one nonzero symbol, or the minimum max distance between two nonzero symbols is $d \leq M$, we can pick a suitable $x \in A_d$.

Suppose then that $d > M$, so that $f(d) < 2^{\sqrt{d}}$.
We choose $n = \lfloor d/2 \rfloor$ and pick an arbitrary width-$n$ grid as the base grid, which produces no adjacent base grid cells with nonzero symbols.

On the binary counter layer we can choose $k \geq \sqrt{5n}$ and pick the initial value of the counter of each row and column to be $0$.
This produces a valid configuration on the layer, and the maximum counter value among all base grid cells is exactly $N$.

Consider now the distance layer.
The common value $q$ of the distance counters must satisfy $1 \leq q \leq 2n+1$, and we claim that we can choose $q = d$.
Note that since $n = \lfloor d/2 \rfloor$, we have $d \leq 2n+1$.

Let $(i,j) \in \Z^2$ be the position of a nonzero symbol.
If the rectangle $R$ defined by $(i,j)$ on the distance layer contains another nonzero symbol on its border, then its height is equal to the max distance $d' \geq d$ between these symbols.
The other signals of the distance layer then perform a comparison that implements the restriction $q \leq d'$.
On the other hand, if $R$ touches the borders of the $7 \times 4$ block defined by $(i,j)$, then no comparison is performed, and no new restrictions are placed on $q$.
Thus we can choose the distance layer so that $q = d$.

Finally, the computation layer is determined by the other layers.
The machine $M$ in any given computation cone receives $q = d$ from the distance counters and a number $b \leq N$ from the binary counters.
For each $k \in \N$ we have $b \leq N \leq f(d) \leq g(d, k) = g(q, k)$, so the simulated machines produce no tiling errors.

\subsection{The SFT $X$ is countable}

The subshifts $X \cap A_n^{\Z^2}$ have already been shown to be countable.
Consider then a configuration $x \in X \setminus \bigcup_{n \leq M} A_n^{\Z^2}$.
The base alphabet layer is in $G_2(A, f)$, which is countable.
If $x$ has a base grid of some finite width $n$, then there are countably many choices for the binary counter layer -- the value of $k$, and the initial counter values of each row and column, both of which sum to less than $2^k$ -- and finitely many choices for the distance layer -- the value $q$ of the distance counter.
The computation layer is determined by the other layers apart from the possible existence of at most two infinite computation cones without south vertices, and a bounded number of counters and/or signals inside them, all of which constitute a countable number of choices.

Suppose then that $x$ does not have a base grid of finite width.
Then it does not have finite secondary grids on the binary counter layer either, and its base alphabet layer has at most four nonzero symbols.
Since we defined all signals within base grid cells (and $7 \times 4$ block thereof) as borders between two regions of distinct colors, a case analysis reveals that there are countably many ways to place the signals of the binary counter layer and the distance layer inside an infinite base grid cell.
As in the case where $x$ has a finite-width base grid, the computation layer likewise has countably many valid configurations.
This finishes the proof of Theorem~\ref{thm:main2}.

\section{Open problems}

We have shown in Theorem~\ref{thm:main} that if an effectively closed shift space $X \subset A^\Z$ has a computable gap function, then it satisfies the countable cover conditions if and only if its lift is countably covered.
Example~\ref{ex:not-computable-gap} shows that the countable cover conditions do not imply the computability of the gap function, although by Proposition~\ref{prop:lower-sc} they imply lower semicomputability.

\begin{question}
  Do the countable cover conditions characterize the property of the lift being countably covered among all $\Z$-shift spaces that admit a gap function?
\end{question}

We do not even know whether there exist shift spaces with countably covered lifts and uncomputable gap functions.
More concretely, we ask the following.

\begin{question}
  Does the lift of the SFT of Example~\ref{ex:not-computable-gap} admit a countable SFT cover?
\end{question}

As for two-dimensional gap width shifts, the question of their soficness remains open, although we dare pose a conjecture.

\begin{conjecture}
  \label{conj:sofic}
  If $f : \N \to \N$ is upper semicomputable, then $G_2(A,f)$ is a sofic shift.
\end{conjecture}

As for Theorem~\ref{thm:main2}, the upper bound $f(n) < 2^{\sqrt{n}}$ is not optimal, but we do not know whether it can be removed completely.
A negative answer (assuming Conjecture~\ref{conj:sofic} is true) would require a completely new technique for proving the nonexistence of countable covers.

\begin{question}
  If $f : \N \to \N$ is upper semicomputable, is $G_2(A,f)$ a countably covered sofic shift?
\end{question}

\bibliographystyle{plain}
\bibliography{gapcolbib}

\end{document}